\documentclass{amsart}[12pt]

\usepackage{enumerate}
\usepackage[hidelinks]{hyperref}
\usepackage{verbatim}
\usepackage{yfonts} %
\usepackage{amssymb} %
\usepackage{amsthm}
\usepackage{array}
\usepackage{booktabs}%
\usepackage{hhline}%
\usepackage{xy} %
\usepackage{epsfig}%
\usepackage{color}%
\usepackage{upgreek}
\usepackage[english]{babel}
\usepackage{epigraph}%
\usepackage{fancybox}%
\usepackage{shadow}
\usepackage{afterpage}
\usepackage{mathrsfs}
\usepackage{enumitem}
\usepackage{tabularx}
\usepackage{subcaption}
\usepackage{graphicx}
\usepackage{type1cm}
\usepackage{eso-pic}
\usepackage{color}
\usepackage{upgreek}
\usepackage{bigints}
\usepackage{thmtools}
\usepackage{thm-restate}

\newtheorem{theorem}{Theorem}
\newtheorem{proposition}[theorem]{Proposition}

\newtheorem{lemma}[theorem]{Lemma}

\theoremstyle{definition}
\newtheorem{definition}[theorem]{Definition}
\newtheorem{remark}[theorem]{Remark}
\newtheorem{example}[theorem]{Example}

\theoremstyle{plain}
\newtheorem{corollary}[theorem]{Corollary}

\newcommand{\vt}{\vspace{.1cm}}

\newcommand{\R}{\mathbb{R} }
\newcommand{\q}{\mathbb{Q} }

\newcommand{\N}{\mathbb{N} }
\newcommand{\h}{\mathbb{H}}

\newcommand{\s}{\mathbb{S}}

\renewcommand{\rho}{\varrho}
\renewcommand{\theta}{\varTheta}
\renewcommand{\Theta}{\varTheta}
\renewcommand{\Sigma}{\varSigma}
\renewcommand{\Omega}{\varOmega}
\renewcommand{\Lambda}{\varLambda}
\renewcommand{\tau}{\uptau}
\usepackage{amsmath}

\newcommand{\overbar}[1]{\mkern 1.5mu\overline{\mkern-1.5mu#1\mkern-1.5mu}\mkern 1.5mu}

\newcommand{\qr}{\q_\epsilon^n\times\R}
\newcommand{\la}{\langle}
\newcommand{\ra}{\rangle}
\newcommand{\pt}{\partial_t}

\makeatletter
\newcommand{\tpitchfork}{%
  \vbox{
    \baselineskip\z@skip
    \lineskip-.52ex
    \lineskiplimit\maxdimen
    \m@th
    \ialign{##\crcr\hidewidth\smash{$-$}\hidewidth\crcr$\pitchfork$\crcr}
  }%
}
\makeatother

\begin{document}

\title
{Isoparametric Hypersurfaces of  $\h^n\times\R$ and
$\s^n\times\R$}
\author{Ronaldo F. de Lima \and Giuseppe Pipoli}
\address[A1]{Departamento de Matem\'atica - Universidade Federal do Rio Grande do Norte,
Lagoa Nova - 59072-970.}
\email{ronaldo.freire@ufrn.br}
\address[A2]{Department of Information Engineering,
Computer Science and Mathematics, Università degli Studi
dell’Aquila, via Vetoio, Coppito - 67100.}
\email{giuseppe.pipoli@univaq.it}

\maketitle

\begin{abstract}
We classify the isoparametric hypersurfaces and the homogeneous
hypersurfaces of $\h^n\times\R$
and $\s^n\times\R$, $n\ge 2$,  by establishing that any such
hypersurface has constant angle function and
constant principal curvatures.

\vspace{.15cm}

\vspace{.1cm}

\end{abstract}

\section{Introduction}
The construction and classification of isoparametric hypersurfaces
in general Riemannian manifolds
has become a matter of great interest in submanifold theory since
the early works by Cartan on this subject.
In a series of four remarkable papers~\cite{cartan1}--\cite{cartan4},
published in the late 1930's, he classified the isoparametric
hypersurfaces of the $n$-hyperbolic spaces $\h^n$, and brought to light that
the classification of the isoparametric hypersurfaces of the $n$-spheres
$\s^n$ is a rather intricate problem. As a matter of fact,
this classification has been built through many works over the last decades,
and only recently it was announced to be complete (see~\cite{cecil,chi}).

The theory of isoparametric hypersurfaces  connects with many branches of
phys\-ics and mathematics. Recently,
it has been applied to establish existence and classification results
to various extrinsic geometric flows
in Riemannian manifolds, such
as the mean curvature flow, the inverse mean curvature flow, elliptic Weingarten flows, and
higher order mean curvature flows
(see~\cite{alencar-tenenblat, delima, delimaW, delima-pipoli, GSS, reis-tenenblat}).

By definition, an isoparametric hypersurface has constant mean curvature,
as do any nearby (locally defined) hypersurface which is parallel to it.
Homogeneous hypersurfaces, that is, those which are codimension one orbits
of isometric actions on the ambient space, are well known
examples of isoparametric hypersurfaces.
It turns out that all isoparametric hypersurfaces of Euclidean space $\R^n$
(classified by Segre~\cite{segre}),
as well as those of hyperbolic space $\h^n$
(classified by Cartan~\cite{cartan1}), are homogeneous.
In  $\s^n,$ one has the two types, homogeneous (classified by Hsiang and Lawson~\cite{HL})
and nonhomogeneous (cf.~\cite{FKM,ozeki-takeuchi}).

Besides being isoparametric, homogeneous hypersurfaces have
constant principal curvatures. However, the constancy of the
principal curvatures implies neither being isoparametric nor homogeneous.
For instance, in simply connected space forms,
as proved by Cartan, a hypersurface is isoparametric if and only if it has
constant principal curvatures. On the other hand, as we pointed out, there are
nonhomogeneous isoparametric hypersurfaces in $\s^n$. Also,
Rodríguez-Vázquez~\cite{avazquez} showed that, for each $n\ge 3$,
there exists an $n$-dimensional torus which
contains a nonisoparametric hypersurface whose principal curvatures are constant.
In addition,  Guimarães, Santos and Santos~\cite{GSS} applied the theory of
mean curvature flow to obtain
a class of Riemannian manifolds that admit nonisoparametric
hypersurfaces with constant principal curvatures.
It should also be mentioned that, in certain Riemannian manifolds,
such as Damek-Ricci spaces,
there exist isoparametric hypersurfaces
whose principal curvatures are not constant functions (see~\cite{ramos-vazquez,geetal}).

In~\cite{dominguezvazquez-manzano}, Domínguez-Vázquez and Manzano
classified the isoparametric surfaces of all simply connected homogeneous
3-manifolds with 4-dimensional isometry
group. These are the so called $\mathbb E(\kappa,\tau)$ spaces with $\kappa-4\tau^2\ne 0$, 
which include the products $\h^2\times\R$ and $\s^2\times\R$.
Their result also provides  the classification
of the homogeneous surfaces of these spaces, as well as
of the surfaces having constant principal curvatures.
Other  results on  classification of isoparametric or
constant principal curvature
hypersurfaces of Riemannian manifolds of nonconstant
sectional curvature were
obtained in~\cite{{chaves},{ramosetal},{vazquez},{vazquez-gorodski},manfio,{santos}}.

Inspired by  Domínguez-Vázquez and Manzano's work,
we aim to establish here the following classification result for
hypersurfaces of the products  $\qr$, where $\q_\epsilon^n$
stands for the $n(\ge\hspace{-.1cm}2)$-dimensional  simply connected space form of
constant sectional curvature $\epsilon=\pm 1$, that is,
the hyperbolic space $\h^n$ for $\epsilon=-1$,
and the sphere $\s^n$ for $\epsilon=1$:

\begin{restatable}{thm}{main} \label{thm-main}
Let $\Sigma$ be a connected hypersurface of $\mathbb Q_\epsilon^n\times\R$.
Then the following are equivalent:
\begin{enumerate}[parsep=1ex]
  \item[\rm (i)] $\Sigma$ is isoparametric.
  \item[\rm (ii)] $\Sigma$ has constant angle and constant principal curvatures.
  \item[\rm (iii)] $\Sigma$ is an open subset of one of the following complete hypersurfaces:
  \begin{itemize}[parsep=1ex]
    \item[\rm (a)] a horizontal slice $\mathbb Q_\epsilon^n\times\{t_0\}$,
    \item[\rm (b)] a vertical cylinder over a complete isoparametric hypersurface of $\mathbb Q_\epsilon^n$,
    \item[\rm (c)] a parabolic bowl of $\mathbb H^n\times\R$ (see Fig.~\ref{fig-parabolichelicoid}).
  \end{itemize}
\end{enumerate}
Moreover, in the hyperbolic case $\epsilon=-1$, the condition
\begin{enumerate}
  \item[\rm (iv)] $\Sigma$ is an open subset of a homogeneous hypersurface
\end{enumerate}
is also equivalent to {\rm (i)--(iii)}. In the spherical case $\epsilon=1$, {\rm (iv)} is equivalent to
\begin{enumerate}
  \item[\rm (v)] $\Sigma$ is an open subset of either a horizontal slice or a vertical cylinder
  over a complete homogeneous hypersurface of \,$\mathbb S^n$.
\end{enumerate}
\end{restatable}

\begin{figure}[hbt]
 \centering
 \includegraphics[width=3cm]{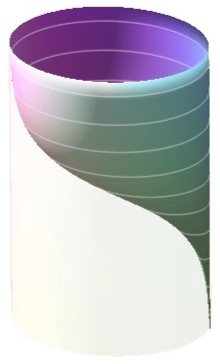}
 \caption{\small The depicted hypersurface, called a \emph{parabolic bowl},
 is a homogeneous entire vertical graph in $\h^n\times\R$
 whose level hypersurfaces
 are parallel horospheres, and whose vertical translations define a
 nonsingular isoparametric foliation of $\h^n\times\R$
 (we are grateful  to João P. dos Santos for this plot).}
 \label{fig-parabolichelicoid}
\end{figure}

The most delicate part of the proof of Theorem~\ref{thm-main}, which
we do in Proposition~\ref{prop-isoparametric-->constantangle},  is
showing that the connected isoparametric hypersurfaces of $\qr$ have
constant angle function. To accomplish that, we proceed as
Domínguez-Vázquez and Manzano in~\cite{dominguezvazquez-manzano}. More precisely,
we apply Jacobi field theory for reducing the proof to
the resolution of an algebraic problem.
In the $n$-dimensional setting, such problem is considerably
more involved than its $2$-dimensional analogue,
which compelled us to approach it differently.
Our trick then was to consider an alternate formulation on
which the corresponding
algebraic equations are all linear. In this way, the
solution became attainable,
although this linear problem were still arduous
(see Section~\ref{sec-appendix}).

For the remaining of the proof of Theorem~\ref{thm-main},
we apply some results obtained in~\cite{delima-roitman,tojeiro},
including the one that  characterizes
constant angle  hypersurfaces of
the products $\qr$ as horizontal slices,
vertical cylinders or vertical graphs built
on parallel hypersurfaces of $\mathbb Q_\epsilon^n$
(cf.~Sec.~\ref{subsec-graphs}).

Theorem~\ref{thm-main} provides  an
explicit classification of the isoparametric hypersurfaces
and of the homogeneous hypersurfaces of $\qr$, since such
classes of hypersurfaces
are completely classified in
$\mathbb Q_\epsilon^n$\footnote{Due to some controversial
results by Siffert~\cite{siffert1,siffert2},
there is no general agreement that the isoparametric hypersurfaces of
the sphere $\s^{13}$ are indeed classified.}.
The case $n=2$, of course,
is contained in the main result
by Domínguez-Vázquez and Manzano~\cite{dominguezvazquez-manzano}.
We included it here due to the fact that
our proof differs from theirs in some substantial parts. For $n=3$,
it was proved in~\cite{manfio} that  hypersurfaces
with constant principal curvatures have
constant angle. Considering this result, we can
drop the assumption on the
constancy of the angle in Theorem~\ref{thm-main}
when $n=3$. Finally, we remark that
Theorem~\ref{thm-main} also gives that all
isoparametric hypersurfaces of $\qr$ have constant
scalar curvature, since
the scalar curvature of a hypersurface of $\qr$
satisfies a relation (see, e.g.,
equality~(6.2) in~\cite{delima-ramos-santos})
which implies that any such hypersurface
having constant angle and constant principal
curvatures is necessarily of constant scalar curvature.

\vt
\noindent
{\bf Acknowledgments.} This work was partially
supported by INdAM-GNSAGA and PRIN 20225J97H5. It was conceived
during the time the first named author visited
the Department of Information Engineering, Computer Science and
Mathematics of Università degli Studi
dell’Aquila. He would like to deeply express his gratitude for
the warm hospitality provided by the faculty
of that institution.
The second named author was partially supported by INdAM - GNSAGA Project,
codice CUP E53C24001950001.


\section{Preliminaries}  \label{sec-preliminaries}

\subsection{Tensor Curvature of $\qr$}
Given a field  $X\in\mathfrak X(\qr)$,  we shall denote by $X^h$ its component which is
tangent to $\mathbb Q_\epsilon^n$ (called \emph{horizontal}), that is,
\[
X^h:=X-\la X,\pt\ra\partial_t,
\]
where $\partial _t$ denotes the gradient of the projection $\pi_{\scriptscriptstyle\R}$ of
$\qr$ on its second factor $\R.$ Note that $\partial_t$ is a unit parallel field on $\qr.$
When $X=X^h$, we say that the field $X$ is \emph{horizontal}.

Since the first factor of $\qr$ has constant sectional curvature $\epsilon$ and its
second factor  is one-dimensional, the  Riemannian curvature tensor $R_\epsilon$
of $\qr$ is
\begin{equation*}
R_\epsilon(X,Y)Z=\epsilon(\la X^h,Z^h\ra Y^h-\la Y^h,Z^h\ra X^h) \,\,\,\,\, \forall X, Y, Z\in\mathfrak X(\qr).
\end{equation*}

\subsection{Hypersurfaces of $\qr$}
Given an oriented hypersurface $\Sigma$ of $\qr$ (endowed with the
standard product metric),
set $N$ for its unit normal field  and $A$ for its shape operator  with respect to
$N,$ that is
\[
AX=-\overbar\nabla_XN,  \,\,\,  X\in \mathfrak X(\Sigma),
\]
where $\overbar\nabla$ is the Levi-Civita connection of $\qr,$
and $\mathfrak X(\Sigma)$ is the tangent bundle of $\Sigma$.
The principal curvatures of $\Sigma,$ i.e., the eigenvalues of
the shape operator
$A,$ will be denoted by $k_1\,, \dots ,k_n$. In this setting,
the non-normalized \emph{mean curvature} $H$ of  $\Sigma$
is defined as the sum of its principal curvatures, that is,
\[
H:=k_1+\cdots +k_n.
\]

The \emph{height function} $\phi$ and the \emph{angle function} $\Theta$
of $\Sigma$ are defined by
\[
\phi:=\pi_{\scriptscriptstyle\R}|_\Sigma, \quad \Theta:=\langle N,\partial_t\rangle.
\]

Denoting by $\nabla $ the gradient  on $C^\infty(\Sigma)$ and writing
$T:=\nabla\phi,$ the identity
\begin{equation*}
T=\partial_t-\theta N
\end{equation*}
holds everywhere on $\Sigma.$ In particular,
$\|T\|^2=1-\theta^2.$

\subsection{Graphs on parallel hypersurfaces} \label{subsec-graphs}
Let $\mathscr F:=\{M_s\subset\q_\epsilon^n\,;\, s\in I\}$
be a family of parallel hypersurfaces of $\q_\epsilon^n,$ where $I\subset\R$ is an open interval.
Given a smooth function $\phi$ on $I,$ let
$$f\colon M_{s_0}\times I\rightarrow\qr, \,\,\, s_0\in I,$$
be the immersion given by
\begin{equation*}
f(p,s):=(\exp_p(s\eta_{s_0}(p)),\phi(s)), \,\,\, (p,s)\in M_{s_0}\times I,
\end{equation*}
where $\exp$ denotes the exponential map of $\q_\epsilon^n,$ and $\eta_{s_0}$ is the unit
normal of $M_{s_0}.$ The hypersurface $\Sigma=f(M_{s_0}\times I)$ is a vertical graph
over an open set of $\q_\epsilon^n$ whose level hypersurfaces
are the parallels $M_s$ to $M_{s_0}.$

\begin{definition}
With the above notation, we call $\Sigma$ an $(M_s,\phi)$-\emph{graph} and say that
\begin{equation}\label{eq-rho}
\rho(s):=\frac{\phi'(s)}{\sqrt{1+(\phi')^2(s)}},  \,\,\, s\in I,
\end{equation}
is the $\rho$-\emph{function} of $\Sigma$.
\end{definition}

Given an $(M_s,\phi)$-graph $\Sigma$ in $\qr$, it is easily seen that
$N=-\rho(s)\eta_s(p)+\theta\partial_t$ is a unit normal to it.
In particular, one has that the equality
\begin{equation} \label{eq-theta}
\rho^2+\theta^2=1
\end{equation}
holds everywhere on  $\Sigma$.

It was proved in~\cite{delima-manfio-santos} that,
with this orientation, the principal curvature functions
$k_i=k_i(p,s)$ of  $\Sigma$
at a point $(p,s)\in M_{s_0}\times I$ are:
\begin{equation*}
k_i=-\rho(s)k_i^s(p), \,\,\, i\in\{1,\dots, n-1\}, \quad\text{and}\quad k_n=\rho'(s),
\end{equation*}
where $k_i^s(p)$ is the $i$-th principal curvature of the parallel $M_s$  at
$\exp_p(s\eta_{s_0}(p)).$

We point out that, from~\eqref{eq-rho}, one has
\[
\phi(s)=\int_{s_0}^{s}\frac{\rho(u)}{\sqrt{1-\rho^2(u)}}du+\phi(s_0), \,\,\, s_0, s\in I.
\]
Hence, an $(M_s,\phi)$-graph is determined by its $\rho$-function.


\section{Isoparametric hypersurfaces of constant angle} \label{sec-constantangle01}

In this section, we prove Propositions~\ref{prop-parabolichelicoid}
and~\ref{prop-constantangle}, which establish
the existence of the parabolic bowl as described in Figure~\ref{fig-parabolichelicoid}, and
its uniqueness --- together with horizontal slices and vertical
cylinders over isoparametric hypersurfaces of $\q_\epsilon^n$ --- as a hypersurface
of constant angle and constant principal curvatures in $\qr$. In the proofs,
the following two lemmas will play a crucial role. They first appeared in~\cite{tojeiro}, and
later in~\cite{delima-roitman} in a more general form. For notational purposes, we
refer to the corresponding results of~\cite{delima-roitman}.

\begin{lemma}[Theorem~7 of~\cite{delima-roitman}] \label{lem-parallel}
Let  $\{M_s\,;\, s\in I\}$ be a family of  isoparametric
hypersurfaces of \,$\q_\epsilon^n$. Consider  the first order
differential equation
\begin{equation}  \label{eq-ode}
y'=H^sy+H, \,\,\, s\in I,
\end{equation}
where $H^s$ denotes the mean curvature of $M_s$ and $H$ is a constant.
In this setting, if  $\rho\colon I\to(0,1)\subset\R$ is a solution to
\eqref{eq-ode},
then the $(M_s,\phi)$-graph  determined by $\rho$ has constant mean curvature $H$.
Conversely, if an $(M_s,\phi)$-graph $\Sigma$ has constant  mean curvature $H$,
then $\{M_s\,;\, s\in I\}$ is isoparametric and the $\rho$-function
of $\Sigma$ is a solution to~\eqref{eq-ode}.
\end{lemma}

\begin{lemma}[Corollary 4 of~\cite{delima-roitman}] \label{lem-constantangle}
Let $\Sigma$ be a connected hypersurface of \,$\qr$ whose angle function is constant.
Then, one of the following occurs:
\begin{itemize}[parsep=1ex]
\item $\Sigma$ is an open set of a horizontal slice $\q_\epsilon^n\times\{t_0\}$;
\item $\Sigma$ is a vertical cylinder over a hypersurface of \,$\q_\epsilon^n$;
\item $\Sigma$ is locally an $(M_s,\phi)$-graph {with $\phi'$ a nonzero constant}.
\end{itemize}
\end{lemma}

Now, we are in a position to state and prove the announced
Propositions~\ref{prop-parabolichelicoid} and~\ref{prop-constantangle}.

\begin{proposition} \label{prop-parabolichelicoid}
Given $H\in(0,n-1)$, there exists an entire $(M_s,\phi)$-graph
$\Sigma_H$ in $\h^n\times\R$ (to be called a \emph{parabolic bowl}) of constant mean
curvature $H,$ which is homogeneous  and has constant angle. In particular,
$\Sigma_H$ is isoparametric and has constant principal curvatures.
\end{proposition}
\begin{proof}
Let $\mathscr F:=\{M_s\,;\, s\in\R\}$ be a family of parallel horospheres of $\h^n$. Then,
considering the ``outward orientation'' in each $M_s$, for all $s\in\R$ and all $p\in M_s$, we have that
$k_i^s(p)=-1$ for all $i\in{1,\dots, n-1}$. Therefore,
given $H\in(0,n-1)$,
the ODE~\eqref{eq-ode} associated to $\mathscr F$ and $H$ is
\begin{equation}  \label{eq-ode02}
y'=-(n-1)y+H.
\end{equation}
Clearly,
the constant function $\rho=H/(n-1)<1$ is a solution to~\eqref{eq-ode02}. Hence, by
Lemma~\ref{lem-parallel}, the entire $(M_s,\phi)$-graph $\Sigma_H$ of $\h^n\times\R$
which is determined by $\rho$ has constant mean curvature $H$. Moreover,
by~\eqref{eq-theta},
$\Sigma_H$ has constant angle function.

Finally, the homogeneity
of $\Sigma_H$ follows from its invariance by the one-parameter group of
parabolic translations of $\h^n$ which fix each horosphere $M_s$ (extended slicewise to $\h^n\times\R$),
as well as by the isometries $\Psi_u\times\tau_u$, $u\in\R,$  where
$$\{\Psi_u\,;\, u\in\R\}\subset{\rm Iso}(\h^n)$$
is the flow defined by the unit normals of the horospheres $M_s$, and
$\tau_u$ is the vertical translation $(p,t)\mapsto (p,t+\rho u)$ of $\h^n\times\R.$
\end{proof}

When $n=2$, our parabolic bowl
$\Sigma_H$ corresponds to the surface $P_{H,-1,0}$
of~\cite{dominguezvazquez-manzano}, and to the surface
$P_H$ of~\cite{leite}. They also appear in~\cite{manfio}, for $n=3$, and
in~\cite{chaves}, for $n\ge 2$. In all of these occurrences,
the hypersurface is under no specific designation, except
in~\cite{dominguezvazquez-manzano}, where they are called
\emph{parabolic helicoids}.
Our chosen nomenclature comes from the extrinsic geometric flows
theory, for it was shown in~\cite{delima-pipoli} that, for each
$r\in\{1,\dots, n-1\}$, there exists a parabolic bowl
$\Sigma_{H_r}$ of constant $r$-th mean curvature $H_r$
such that $H_r=\theta$ on $\Sigma_{H_r}$. As a consequence,
$\Sigma_{H_r}$ is a translating soliton to mean curvature flow
of $r$-th order. In this context, entire graphs with this property
are called \emph{bowl solitons}.
We should also mention that
any parallel to a given parabolic bowl is nothing but a vertical translation of it.
Therefore, analogously to the slices $\h^n\times\{t\}$, $t\in\R$,
the family of all parallels to a parabolic bowl defines a nonsingular
isoparametric foliation of $\h^n\times\R$.

\begin{proposition} \label{prop-constantangle}
Let $\Sigma$ be a connected  hypersurface of \,$\qr$ with constant angle function.  Then,
$\Sigma$ is isoparametric if and only if it has constant principal curvatures.  If so,
$\Sigma$ is  an open set of one of the following hypersurfaces:
\begin{itemize}[parsep=1ex]
  \item[\rm (i)] a horizontal slice \,$\mathbb Q_\epsilon^n\times\{t_0\}$,
  \item[\rm (ii)] a vertical cylinder over a complete isoparametric hypersurface of \,$\mathbb Q_\epsilon^n$,
  \item[\rm (iii)] a parabolic bowl {if $\epsilon=-1$}.
  \end{itemize}
\end{proposition}

\begin{proof}
By Lemma~\ref{lem-constantangle}, $\Sigma$ is an open set of a horizontal slice, a vertical cylinder
over a hypersurface of $\mathbb Q_\epsilon^n$, or is locally an $(M_s,\phi)$-graph such that $\phi'$ is constant.
If the first occurs, we are done. Assume then that $\Sigma$ is a cylinder $\Sigma_0\times\R$, where
$\Sigma_0$ is a hypersurface of $\mathbb Q_\epsilon^n$.  It is easily seen that
$\Sigma$ is isoparametric (resp. has constant principal curvatures) if and only if $\Sigma_0$ is isoparametric
(resp. has constant principal curvatures). However, in $\q_\epsilon^n,$
to be isoparametric and to have constant principal curvatures are equivalent conditions. Besides,
we have from the classification of isoparametric hypersurfaces of space forms that
any isoparametric hypersurface of $\q_\epsilon^n$ is necessarily an open set of a complete
isoparametric hypersurface.

Let us suppose now that $\Sigma$ is locally an $(M_s,\phi)$-graph of $\qr$, $s\in I$,
such that $\phi'$ is a nonzero constant. Then, from~\eqref{eq-rho},
the $\rho$-function of $\Sigma$ is a nonzero constant as well.
If either $\Sigma$ is isoparametric or has
constant principal curvatures, then its mean curvature $H$ is constant.
In this case, it follows from Lemma~\ref{lem-parallel} that the hypersurfaces
$M_s$ are isoparametric
and that $\rho$ satisfies
$0=\rho'=\rho H^s+H,$
which implies that the mean curvature $H^s$ of $M_s$
is a constant independent of $s$. Again, by considering the classification of
isoparametric hypersurfaces
of $\q_\epsilon^n$, one concludes that
$\{M_s\,;\, s\in I\}$ is necessarily a family of parallel horospheres of $\h^n$, so that
$\Sigma$ is an open set of a parabolic {bowl} of $\h^n\times\R$.
\end{proof}

\begin{remark}
The first part of Proposition~\ref{prop-constantangle} is essentially the content
of Corollary~5.2
of~\cite{chaves}. Nonetheless, the proof which is given there is rather distinct from ours.
\end{remark}


\section{The constancy of the angle function}  \label{sec-constantangle02}

This section will be entirely dedicated to the proof of the following proposition,
which constitutes our main result for establishing Theorem~\ref{thm-main}.

\begin{proposition} \label{prop-isoparametric-->constantangle}
The angle function of  any connected isoparametric hypersurface of \,$\qr$ is constant.
\end{proposition}

\begin{proof}[Proof of Proposition~\ref{prop-isoparametric-->constantangle}]
Let $\Sigma$ be a connected oriented isoparametric hypersurface of $\qr$.
Since its angle function $\theta$ is continuous, it suffices to prove that
$\theta$ is locally constant on $\Sigma$. We can assume that
$\theta^2\ne1$, so that the gradient $T$ of the height function of $\Sigma$
never vanishes. In this setting, define
$\Phi^r\colon\Sigma\to\qr$ by
\[
\Phi^r(p)=\exp_p(rN_p),
\]
where $N_p$ is the unit normal field of $\Sigma$
at $p\in\Sigma$ and $\exp$ stands for
the exponential map of $\qr$. Passing to an open subset of
$\Sigma$, we can assume that, for a
small $\delta>0,$ and for all $r\in(-\delta,\delta),$
the map $\Phi^r$ is well defined and $\Sigma_r:=\Phi^r(\Sigma)$ is an embedded hypersurface
of $\qr$ which lies at distance $|r|$ from $\Sigma$.

Given $p\in\Sigma$,
let $\gamma_p(r)$ be the geodesic of $\qr$ such that
$\gamma_p(0)=p$ and $\gamma_p'(0)=N_p$, that is,
$\gamma_p(r)=\Phi^r(p)$. It is easily seen that
the unit normal to $\Sigma^r$ at $\gamma_p(r)$ is
$N(r):=\gamma_p'(r)$. In particular,
$N(r)$ is parallel along $\gamma_p$. Since
$\partial_t$ is parallel on $\qr$, this gives that
the angle function of $\Sigma_r$ is constant along
$\gamma_p$. Consequently, the gradient of the height function
of $\Sigma_r$ at $\gamma_p(r)$ is the parallel transport
$T(r)$ of $T$ along $\gamma_p$.

Set  $U_1(r)=T(r)/|T(r)|$ and let $N(r),U_1(r),U_2(r),\dots,U_n(r)$
be an orthonormal parallel frame along $\gamma_p$.
Notice that, for all $i\in\{2,\dots, n\}$, $U_i(r)$ is horizontal. Indeed, for such an $i$,
$0=\langle U_i(r),T(r)\rangle=\langle U_i,\partial_t\rangle.$

For any $j\in\{1,\dots,n\}$, let $\zeta_j=\zeta_j(r)$ be the Jacobi field along $\gamma_p$ such
that
\begin{equation} \label{eq-initialconditions}
\zeta_j(0)=U_j(0) \quad\text{and}\quad \zeta'_j(0)=-AU_j(0),
\end{equation}
where $A$ is the shape operator of $\Sigma$.
Then, for any such $j$, $\zeta_j$ satisfies
\begin{equation} \label{eq-jacobi}
\zeta_j''+R_\epsilon(\gamma_p',\zeta_j)\gamma_p'=0.
\end{equation}
In addition, for all  $r\in(-\delta,\delta),$
one has
$\la\zeta_j(r),N(r)\ra=0$.  Therefore,
there exist smooth functions $b_{ij}=b_{ij}(r)$ such that
\begin{equation} \label{eq-zetacoordinates}
\zeta_j=\sum_{i=1}^nb_{ij}U_i, \,\,\, j\in\{1,\dots, n\}.
\end{equation}
Furthermore, since all $U_i$ are parallel along
$\gamma_p$, we have
\begin{equation}\label{eq-zeta''}
\zeta''_j=\sum_{i=1}^nb_{ij}''U_i, \,\,\, j\in\{1,\dots, n\}.
\end{equation}

Now, aiming at the Jacobi equation~\eqref{eq-jacobi}, we compute
\begin{equation}  \label{eq-jacobi02}
R_\epsilon(\gamma_p', \zeta_j)\gamma_p'=R_\epsilon(N,\zeta_j)N=
\epsilon(\la N^h,N^h\ra \zeta_j^h-\la \zeta_j^h,N^h\ra N).
\end{equation}
However, $N^h=N-\theta\partial_t,$ so that $\la N^h,N^h\ra=1-\theta^2=\|T\|^2$.
Besides, considering~\eqref{eq-zetacoordinates} and the fact  that $U_i$ is horizontal for all $i\geq 2$, we have:
\[
\zeta_j^h=\zeta_j-\la\zeta_j,\partial_t\ra\partial_t=\zeta_j-b_{1j}\la U_1,\partial_t\ra\partial_t=\zeta_{j}-b_{1j}\|T\|\partial_t.
\]
In particular,
\[
\la \zeta_j^h,N^h\ra=\la\zeta_j-b_{1j}\|T\|\partial_t,N-\theta\partial_t\ra=-\theta\la\zeta_j,\partial_t\ra=-\theta b_{1j}\|T\|.
\]

It follows from~\eqref{eq-jacobi02} and the above equalities that,
\begin{eqnarray*}
\epsilon R_\epsilon(\gamma_p',\zeta_j)\gamma_p' &=&\|T\|^2\zeta_j-b_{1j}\|T\|\partial_t+\theta b_{1j}\|T\|N=
\|T\|^2\zeta_j-b_{1j}\|T\|T\\
&=& \|T\|^2(\zeta_j-b_{1j}U_1)=\sum_{i=2}^{n}\|T\|^2b_{ij}U_i.
\end{eqnarray*}
This last equality and~\eqref{eq-jacobi} then give
\begin{equation} \label{eq-zeta''02}
\zeta_j''=\sum_{i=2}^{n}-\epsilon\|T\|^2b_{ij}U_i.
\end{equation}

Now, set $(a_{ij})$ for the (symmetric) matrix of $A$ with respect to
the orthonormal basis $\{U_1(0),\dots, U_n(0)\}$, that is,
$$AU_j(0)=\sum_{i=1}^na_{ij}U_i(0).$$

Considering this last equality and comparing~\eqref{eq-zeta''} with~\eqref{eq-zeta''02}, we conclude
that $\zeta_j$ is a Jacobi field satisfying the initial conditions~\eqref{eq-initialconditions}
if and only if the coefficients
$b_{ij}$ are solutions of an initial value problem. Namely,
\begin{equation}\label{eq-IVP}
\left\{\begin{array}{rcll}
b_{1j}''(r)&=&0
;\\[1ex]
b_{ij}''(r)&=&-\epsilon\|T\|^2b_{ij}(r),&\text{if}\ i\geq 2;
\\[1ex]
b_{ij}(0)&=&\delta_{ij}
;\\[1ex]
b_{ij}'(0)&=&-a_{ij}
.
\end{array}\right.
\end{equation}

Defining the function
$$\tau:=-\epsilon\|T\|^2=-\epsilon(1-\theta^2)\ne 0,$$
the solutions of~\eqref{eq-IVP} are
\begin{equation}\label{eq-solutions}
\left\{\begin{array}{rcll}
b_{1j}(r)&=&\delta_{1j}-a_{1j}r
;\\[1ex]
b_{ij}(r)&=&\delta_{ij}c_\tau(r)-a_{ij}s_\tau(r)&\text{if}\ i\geq 2
;
\end{array}\right.
\end{equation}
where $s_\tau$ and $c_\tau$ are the functions:
\begin{equation}\label{def-sc}
s_\tau(r):=\left\{
\begin{array}{lcc}
\frac{1}{\sqrt{\tau}}\sinh(\sqrt\tau\, r) & \text{if} & \tau>0,\\[1ex]
\frac{1}{\sqrt{-\tau}}\sin(\sqrt{-\tau}\, r)  & \text{if} & \tau<0,
\end{array}
\right.
\quad
c_\tau(r):=\left\{
\begin{array}{lcc}
\cosh(\sqrt\tau\, r) & \text{if} & \tau>0,\\[1ex]
\cos(\sqrt{-\tau}\, r)  & \text{if} & \tau<0.
\end{array}
\right.
\end{equation}

Notice that the derivatives of  $s_\tau$ and $c_\tau$
satisfy:
\begin{equation} \label{eq-s'&c'}
s_\tau'(r)=c_\tau(r) \quad\text{and}\quad c_\tau'(r)=\tau s_\tau(r) \,\,\, \forall r\in\R.
\end{equation}

Given $r\in(-\delta,\delta),$ let
$B(r)$ and $C(r)$ be the linear operators of $T_{\gamma_p(r)}\Sigma_r$ which take
the basis $\{U_1(r), \dots, U_n(r)\}$ to
$\{\zeta_1(r), \dots, \zeta_n(r)\}$ and $\{\zeta_1'(r), \dots, \zeta_n'(r)\},$
respectively. Considering~\eqref{eq-zetacoordinates} and the fact that each
$U_i$ is parallel along $\gamma_p$, we conclude that their matrices with respect to this basis are
$$B(r)=(b_{ij}(r))\quad\text{and}\quad C(r)=(b_{ij}'(r)), \,\,\, i,j\in\{1,\dots, n\},$$
where $b_{ij}$ are the functions defined in~\eqref{eq-solutions}.

In the above setting, Jacobi field theory applies, giving that
$B(r)$ is nonsingular for each $r\in(-\delta,\delta)$, and that the
shape operator  of $\Sigma^r$ is $A^r=-C(r)B(r)^{-1}$ (see~\cite[Theorem 10.2.1]{olmos}).
In particular,
\[
H(r)={\rm trace}\,A^r=-{\rm trace}\,(C(r)B(r)^{-1})=-\frac{\frac{d}{dr}(\det B(r))}{\det B(r)},
\]
where, in the last equality, we used the fact that $\nabla\det (B)=(\det B){\rm adj}(B^{-1}).$

Defining $D(r)=\det B(r)$,  it follows from the above that the function
\[
f(r)=D'(r)+H(r)D(r)
\]
vanishes identically.
Since $D'=f-HD=-HD$, one has
\[
f'=D''+H'D+HD'=D''+(H'-H^2)D.
\]
Therefore,  for all $k\in\N$,
\begin{equation} \label{eq-fderivatives}
0=f^{(k)}=D^{(k+1)}+\phi_kD, \quad \phi_k=\phi_k(H,H',\dots, H^{(k)}),
\end{equation}
where $f^{(k)}$ denotes the $k$-th derivative of $f$.

Now, considering the equalities~\eqref{eq-solutions},
we write the matrix $B(r)=(b_{ij}(r))$  as
\begin{equation}\label{eq-B(r)}
B(r)=
\left[\begin{array}{c|c}
1-a_{11}r& \begin{array}{ccc}
-a_{12}r&\dots&-a_{1n}r
\end{array}\\
\hline
\begin{array}{c}
-a_{12}s_\tau(r)\\
\vdots\\
-a_{1n}s_\tau(r)
\end{array}& \delta_{ij}c_\tau(r)-a_{ij}s_\tau(r)
\end{array}\right].
\end{equation}

Set
\[
d_k=-\phi_k(H(0),H'(0),\dots, H^k(0)), \quad k\in\N.
\]

The following lemma will play a fundamental role in our final argument.


\begin{lemma} \label{lem:correction}
Suppose that \(n\) is odd. Then, for every positive integer \(s_0\),
there exists \(s_*\geq s_0\) such that \(d_{s_*}\neq 0\).
\end{lemma}

\begin{proof}

Assume, by contradiction, that
\[
d_s=0
\quad \forall s\geq s_0.
\]
Then, by~\eqref{eq-fderivatives},
\begin{equation} \label{eq:Ds=0}
D^{(s+1)}(0)=0
\quad \forall s\geq s_0.
\end{equation}

Fix an arbitrary point \(p\in\Sigma\). Since the functions
$c_\tau$ and  $s_\tau$ are real analytic, so is $D$. Then,
it follows from~\eqref{eq:Ds=0} that \(D=D(r)\) agrees with a polynomial in a
neighborhood of \(r=0\).

First, we shall classify the symmetric matrices $A$
for which $D$ is a polynomial in $r$,
then we will show that none of them can be the second fundamental form
of an isoparametric hypersurface.

Write the matrix in~\eqref{eq-B(r)} in block form as
\[
B(r)=
\begin{bmatrix}
1-ar & -rv^T\\[1ex]
-s_\tau(r)v & c_\tau(r)I_m-s_\tau(r)A_0
\end{bmatrix},
\]
where \(m=n-1\) and
\[
a=a_{11},
\qquad
v=(a_{12},\ldots,a_{1n})^T,
\qquad
A_0=(a_{ij})_{2\leq i,j\leq n}.
\]
Set
\[
M(r)=c_\tau(r)I_m-s_\tau(r)A_0,
\qquad
\varphi(r)=\det M(r).
\]

The standard determinant formula for block matrices, together with
the matrix determinant lemma, gives
\begin{equation} \label{eq:Dr&var&psi}
D(r)=\varphi(r)-r\psi(r) \quad\text{and}\quad
\psi(r)=a\varphi(r)+s_\tau(r)v^T\operatorname{adj}(M(r))v.
\end{equation}

We claim that
\begin{equation} \label{eq:var&psi=ct}
\varphi\equiv 1
\quad\text{and}\quad
\psi\equiv a.
\end{equation}

Indeed, let \(\lambda_1,\ldots,\lambda_m\) be the eigenvalues of the
symmetric matrix \(A_0\), and let
$\widetilde v=(\widetilde v_1,\ldots,\widetilde v_m)$
be the coordinates of \(v\) with respect to an orthonormal eigenbasis
of \(A_0\). Then
\begin{equation} \label{eq:var&prod}
\varphi(r)=\prod_{\ell=1}^{m}\bigl(c_\tau(r)-\lambda_\ell s_\tau(r)\bigr),
\end{equation}
whereas
\begin{equation} \label{stau&prod}
s_\tau(r)v^T\operatorname{adj}(M(r))v=\sum_{\ell=1}^{m}\widetilde v_\ell^{\,2}s_\tau(r)
\prod_{\substack{k=1\\k\neq\ell}}^{m}\bigl(c_\tau(r)-\lambda_k s_\tau(r)\bigr).
\end{equation}

Set
\[
\omega=
\begin{cases}
\sqrt{\tau},&\tau>0,\\[2mm]
\sqrt{-\tau},&\tau<0.
\end{cases}
\]
For \(\tau<0\), we have
\begin{equation} \label{eq:ctaustau001}
c_\tau(r)
=\frac12\left(e^{i\omega r}+e^{-i\omega r}\right),
\qquad
s_\tau(r)=
\frac{1}{2i\omega}
\left(e^{i\omega r}-e^{-i\omega r}\right),
\end{equation}
whereas, for \(\tau>0\),
\begin{equation} \label{eq:ctaustau002}
c_\tau(r)
=
\frac12\left(e^{\omega r}+e^{-\omega r}\right),
\qquad
s_\tau(r)
=
\frac{1}{2\omega}
\left(e^{\omega r}-e^{-\omega r}\right).
\end{equation}

It follows from~\eqref{eq:var&prod}--\eqref{eq:ctaustau002}
that there exist constants \(\alpha_j,\beta_j\)
such that
\[
\varphi(r)=\sum_{j=0}^{m}\alpha_je^{\nu_jr},
\qquad
\psi(r)=\sum_{j=0}^{m}\beta_je^{\nu_jr},
\]
where
\[
\nu_j=
\begin{cases}
(m-2j)\sqrt{\tau},&\tau>0,\\[2mm]
i(m-2j)\sqrt{-\tau},&\tau<0.
\end{cases}
\]
Note that  \(\nu_0,\ldots,\nu_m\) are pairwise distinct. Moreover,
since \(m=n-1\) is even,
\[
\nu_j=0
\quad\Longleftrightarrow\quad
j=\frac m2\cdot
\]

Using \(D=\varphi-r\psi\), we obtain
\[
0=
\sum_{\substack{0\leq j\leq m\\j\neq m/2}}
(\alpha_j-r\beta_j)e^{\nu_jr}
+
\left(
\alpha_{m/2}-r\beta_{m/2}-D(r)
\right).
\]

Recall the standard uniqueness property of exponential polynomials:
if \(\mu_1,\ldots,\mu_\ell\) are pairwise distinct and
\[
\sum_{j=1}^{\ell}P_j(r)e^{\mu_jr}=0
\]
on an open interval, where the \(P_j\) are polynomials, then
\(P_j=0\) for every \(j\). Since \(D\) is a polynomial near \(r=0\),
this property applied to the preceding equality yields
\[
\alpha_j=\beta_j=0
\quad \forall j\neq\frac m2, \qquad D(r)=\alpha_{m/2}-r\beta_{m/2}.
\]
Thus \(D\) is affine and
\[
\varphi(r)=\alpha_{m/2},
\qquad
\psi(r)=\beta_{m/2}.
\]
In particular, both \(\varphi\) and \(\psi\) are constant. Since
\[
\varphi(0)=1
\qquad\text{and}\qquad
\psi(0)=a,
\]
we conclude that~\eqref{eq:var&psi=ct} holds.

Suppose first that \(\tau<0\), and set \(z=e^{i\omega r}\). It follows
from~\eqref{eq:var&prod} and~\eqref{eq:ctaustau001} that
\begin{equation}  \label{eq:var=1}
1=\varphi(r)=2^{-m}\prod_{j=1}^{m}\left[\left(1+i\frac{\lambda_j}{\omega}\right)z+
\left(1-i\frac{\lambda_j}{\omega}\right)z^{-1}
\right].
\end{equation}
However, the coefficient of
\(z^m\) in the right-hand side is
\[
2^{-m}\prod_{j=1}^{m}
\left(1+i\frac{\lambda_j}{\omega}\right),
\]
which is nonzero because every \(\lambda_j\) is real. Thus the
right-hand side of~\eqref{eq:var=1} cannot be constant. This contradiction
rules out the case \(\tau<0\).

We must therefore have \(\tau>0\). Set \(z=e^{\omega r}\).
From~\eqref{eq:var&prod} and~\eqref{eq:ctaustau002}, we obtain
\begin{equation} \label{eq:var=1003}
1=\varphi(r)
=
2^{-m}
\prod_{j=1}^{m}
\left[
\left(1-\frac{\lambda_j}{\omega}\right)z
+
\left(1+\frac{\lambda_j}{\omega}\right)z^{-1}
\right].
\end{equation}

The coefficients of $z^{\pm m}$ on the right-hand side
of~\eqref{eq:var=1003} vanish, so
\[
2^{-m}\prod_{j=1}^{m}
\left(1\mp\frac{\lambda_j}{\omega}\right)=0.
\]
Hence, after reordering the indices, we may assume that
$\lambda_1=\omega$ and $\lambda_m=-\omega$. The corresponding
factors are $2z^{-1}$ and $2z$, whose product is $4$. Thus,
\[
1=
2^{-(m-2)}
\prod_{j=2}^{m-1}
\left[
\left(1-\frac{\lambda_j}{\omega}\right)z
+
\left(1+\frac{\lambda_j}{\omega}\right)z^{-1}
\right].
\]
If $m>2$, comparing the coefficients of $z^{\pm(m-2)}$ gives
\[
2^{-(m-2)}
\prod_{j=2}^{m-1}
\left(1\mp\frac{\lambda_j}{\omega}\right)=0.
\]
Therefore, again, after reordering the remaining indices, we may assume
that $\lambda_2=\omega$ and $\lambda_{m-1}=-\omega$.

Since $m$ is even, repeating this argument removes all $m$ factors
in pairs, each corresponding to one eigenvalue $\omega$ and one
eigenvalue $-\omega$. Hence the eigenvalues of $A_0$ are
$\omega$ and $-\omega$, each with multiplicity $m/2$.

Now, by~\eqref{eq:Dr&var&psi} and~\eqref{eq:var&psi=ct},
\[
s_\tau(r)v^T\operatorname{adj}(M(r))v=0.
\]
Differentiating with respect to \(r\), we obtain
\[
c_\tau(r)v^T\operatorname{adj}(M(r))v
+s_\tau(r)v^T\frac{d}{dr}\bigl(\operatorname{adj}(M(r))\bigr)v=0.
\]
Setting $r=0$ then yields
\[
0=v^T\operatorname{adj}(I_m)v=v^Tv=\|v\|^2.
\]
Hence \(v=0\), and therefore \(U_1=T/\|T\|\) is a principal direction.
Moreover,
\begin{equation}  \label{eq:AU1}
AU_1=aU_1
\qquad\text{and}\qquad
\operatorname{spec}\left(A|_{U_1^\perp}\right)
=\{\sqrt{\tau},-\sqrt{\tau}\},
\end{equation}
with
\begin{equation}  \label{eq:AU1001}
\operatorname{mult}_{A|_{U_1^\perp}}(\sqrt{\tau})
=
\operatorname{mult}_{A|_{U_1^\perp}}(-\sqrt{\tau})
=\frac m2\cdot
\end{equation}

Since the point \(p\in\Sigma\) was arbitrary, conclusions
\eqref{eq:AU1}--\eqref{eq:AU1001} hold at every point of \(\Sigma\). The case
\(\tau<0\) having been ruled out, we necessarily have
\(\epsilon=-1\), so that
\[
\Sigma\subset\mathbb H^n\times\mathbb R
\qquad\text{and}\qquad
\tau=\|T\|^2.
\]

Let \(M_c\) be a connected component of a regular slice of \(\Sigma\),
namely,
\[
M_c=\Sigma\cap\bigl(\mathbb H^n\times\{c\}\bigr).
\]
For every \(X\in TM_c=U_1^\perp\), we have
\[
X(\Theta)=-\langle AX,T\rangle=0,
\]
for \(U_1=T/\|T\|\) is a principal direction. Hence \(\Theta\), and
therefore
\[
\|T\|=\sqrt{1-\Theta^2},
\]
is constant along \(M_c\).

A horizontal unit normal to \(M_c\subset\mathbb H^n\) is
\[
\eta=\frac{N-\Theta\partial_t}{\|T\|}\cdot
\]

Let \(S\) denote the shape operator of \(M_c\subset\mathbb H^n\) with
respect to \(\eta\). Since \(\Theta\) and \(\|T\|\) are constant along
\(M_c\), and \(\partial_t\) is parallel, we have, for every
\(X\in TM_c\),
\[
SX
=
-\overline{\nabla}_X\eta
=
\frac{1}{\|T\|}AX
=
\frac{1}{\sqrt{\tau}}AX.
\]

This equality, together with~\eqref{eq:AU1}--\eqref{eq:AU1001}, implies  that \(M_c\) has precisely two
constant principal curvatures,
$1$ and $-1,$
each with multiplicity \(m/2\). Thus \(M_c\) is an isoparametric
hypersurface of
$\mathbb H^n\times\{c\}\simeq\mathbb H^n.$
However, the classification of isoparametric hypersurfaces of
\(\mathbb H^n\) shows that no such hypersurface
exists (cf.~\cite{cartan1},~\cite[Thm.~7]{cecil}). This
contradiction completes the proof of the lemma.
\end{proof}


Expanding $D(r)=\det B(r)$ with respect to the first row of $B(r)$
and considering the equalities~\eqref{eq-s'&c'}, one can easily
prove by induction that, for any integers $n\ge 2$ and  $k\ge 1$,
and for any $\ell\in\{0,\dots, n-1\}$, there are coefficients $\alpha_{\ell,k}, \beta_{\ell,k}$,
which do not depend on $r$,  such that
\begin{equation}\label{eq-Dderivatives}
D^{(k)}(r)=\sum_{\ell=0}^{n-1}(\alpha_{\ell,k}+\beta_{\ell,k}r)s_\tau^\ell(r)c_\tau^{n-1-\ell}(r).
\end{equation}

Taking the first $2n-1$ derivatives of the (constant) function $f$ at $r=0$, and
recalling that $$d_k=-\phi_k(H(0),H'(0),\dots, H^{(k)}(0)),$$
it follows from~\eqref{eq-fderivatives} and~\eqref{eq-Dderivatives} that
\begin{equation}  \label{eq-fderivativesagain}
0=f^{(k)}(0)=\alpha_{0,k+1}-d_k \quad \text{for every}\ k\in\{1,\dots, 2n-1\}.
\end{equation}
In addition, as shown in Lemma~\ref{lem-reduction} of the appendix,
the coefficients $\alpha_{\ell,k}$ and $\beta_{\ell,k}$ satisfy recursive equations which
allow us to express each of them
in terms of $\alpha_{\ell,0}$ and $\beta_{\ell,0}.$
In this way, for any $k\in\N$, after $k+1$ steps,
we can write  $\alpha_{0,k+1}$ as the  linear combination
\begin{equation} \label{eq-linearcombination}
\alpha_{0,k+1}=\sum_{\ell=0}^{n-1}(p_{k+1,\ell}\alpha_{\ell,0}+q_{k+1,\ell}\beta_{\ell,0}),
\end{equation}
where the coefficients $p_{k+1,\ell}$ and $q_{k+1,\ell}$
depend only on $k$, $n$, $\ell$, and $\tau$. Moreover, we have that $\alpha_{0,0}=D(0)=1$ and,
by~\eqref{eq-fderivativesagain}, that $\alpha_{0,k+1}$ coincides with the constant $d_k$.
Therefore,  the vector
\[
x_0=(\alpha_{1,0},\dots,\alpha_{n-1,0},\beta_{0,0},\dots, \beta_{n-1,0})
\]
is a solution to the linear system
\begin{equation} \label{eq-linearsystem}
Mx=P, \quad x\in\R^{2n-1},
\end{equation}
whose augmented matrix $[M|P]$ has vector rows
$L_1, \dots, L_{2n-1}\in\R^{2n}$, where
\[
L_k:=(p_{k+1,1}, p_{k+1,2},\dots, p_{k+1,n-1},q_{k+1,0},q_{k+1,1},\dots ,q_{k+1,n-1},d_k-p_{k+1,0}).
\]

In what follows, by means of  a thorough analysis of the system~\eqref{eq-linearsystem}, we shall show that
$\tau=\tau(p)$ is necessarily a root of an algebraic equation, and so $\tau$ must be constant on $\Sigma$.
To that end, it will be convenient to consider first the cases $n=2$ and $n=3$.

\vspace{.1cm}
\noindent
{\bf Case 1: $n=2$.}
\vspace{.1cm}

As we mentioned before, this case was considered in~\cite{dominguezvazquez-manzano}.
We include it here to better illustrate our strategy, which is distinct from the one employed there.

For $n=2$,  the equalities of Lemma~\ref{lem-reduction} in the appendix yield
\begin{itemize}[parsep=1ex]
\item $\alpha_{0,2}=2\beta_{1,0}+\tau\alpha_{0,0}$;
\item $\alpha_{0,3}=\tau\alpha_{1,0}+3\tau\beta_{0,0}$;
\item $\alpha_{0,4}=4\tau\beta_{1,0}+\tau^2\alpha_{0,0}$.
\end{itemize}

These  equalities imply that
the augmented matrix $[M|P]$ of the system~\eqref{eq-linearsystem} is
\[
[M|P]=\left[
\begin{array}{ccc|c}
0 & 0 & 2 & d_1-\tau\\
\tau & 3\tau & 0 & d_2\\
0 & 0 & 4\tau & d_3-\tau^2
\end{array}
\right],
\]
where $d_i=-\phi_i(0)$, $i\in\{1,2,3\}.$ However,
it is easily seen that $\det M=0$. Hence, denoting
by $M_j$ the matrix obtained from $M$ by replacing its $j$-th column with $P$,
we have that
$\det M_j=0$ for all
$j\in\{1,2,3\}$. Otherwise, by Cramer's rule, the system $Mx=P$ would have no
solution, thereby contradicting the existence of the solution $x_0=(\alpha_{1,0},\beta_{0,0},\beta_{1,0}).$
In particular, we have
\begin{equation} \label{eq-n=2}
0=\det M_2=2\tau^3-4d_1\tau^2+2d_3\tau,
\end{equation}
so that $\tau$ is a root of a polynomial equation. This proves
Proposition~\ref{prop-isoparametric-->constantangle} for $n=2$.

\vspace{.1cm}
\noindent
{\bf Case 2: $n=3$.}
\vspace{.1cm}

Firstly, we point out that, as shown in Proposition~\ref{prop-coeff3} in the appendix, the following
equalities hold:
\begin{itemize}[parsep=1ex]
\item $L_{2s-1}=(0,2^{2s-1}\tau^{s-1},0,s2^{2s-1}\tau^{s-1},0,d_{2s-1}-2^{2s-1}\tau^s)$;
\item $L_{2s}=(2^{2s}\tau^s,0,(2s+1)2^{2s-1}\tau^s,0,(2s+1)2^{2s-1}\tau^{s-1},d_{2s})$.
\end{itemize}

Proceeding as in the previous case, we conclude that, for $n=3$,
the augmented matrix $[M|P]$ is
$${\small
[M|P]=\left[\begin{array}{ccccc|c}
0&2&0&2&0&d_1-2\tau\\
4\tau&0&6\tau&0&6&d_2\\
0&8\tau&0&16\tau&0&d_3-8\tau^2\\
16\tau^2&0&40\tau^2&0&40\tau&d_4\\
0&32\tau^2&0&96\tau^2&0&d_5-32\tau^3\\
\end{array}\right]}.
$$

Again, we have  $\det M=0,$
so that
$\det M_j=0$ for any $j\in\{1,\dots, 5\}$, for
$$x_0=(\alpha_{1,0},\alpha_{2,0},\beta_{0,0},\beta_{1,0},\beta_{2,0})$$
is a solution to~\eqref{eq-linearsystem}.
Since the 3rd and 5th column vectors
of $M$ are linearly dependent, for $j\notin\{ 3, 5\},$ the equality $\det M_j=0$ holds for any
$\tau$. On the other hand, a direct computation gives
\begin{equation} \label{eq-detM3}
\det M_5=-\det M_3=d_12^{14}\tau^6-d_32^{13}\tau^5+d_52^{10}\tau^4.
\end{equation}

Thus, if $d_1, d_3$ and $d_5$ are not all zero,
any of the equalities $\det M_j=0$, $j\in\{3, 5\}$, gives that
$\tau$ is a nonzero root of a polynomial equation.
However, if $d_1=d_3=d_5=0$,  the determinants of all
$M_j$ will vanish identically. To overcome this problem,
we replace the system $Mx=P$ with a suitable system
$\overbar Mx=\overbar P$ as follows.

Given $s>3,$ consider the system $\overbar Mx=\overbar P$, where
$[\overbar M|\overbar P]$ is obtained from $[M|P]$ by replacing its
fifth row vector $L_5$ with $L_{2s-1}$. The above reasoning applied to
$\overbar Mx=\overbar P$ leads to the same conclusion:
if $d_1, d_3$ and $d_{2s-1}$ are not all zero,
then $\tau$ is a nonzero root of a polynomial equation.
So,  assume that
\[
d_k=0 \quad \text{for every odd} \,\,\, k\ge 1.
\]

Then, it follows from
Lemma~\ref{lem:correction} that there exists
$s_*=2s> 4$ such that
\[
d_{2s}\ne 0.
\]

Consider the system $\overbar Mx=\overbar P$, where
$[\overbar M|\overbar P]$ is obtained from $[M|P]$ by replacing its
fifth row vector $L_5$ with $L_{2s}.$ Setting
\[
\lambda=2^{2s}\tau^s, \quad \mu=(2s+1)2^{2s-1}\tau^s \quad\text{and}\quad \nu=(2s+1)2^{2s-1}\tau^{s-1},
\]
$[\overbar M|\overbar P]$ takes the form
$${\small
[\overbar M|\overbar P]=\left[\begin{array}{ccccc|c}
0&2&0&2&0&d_1-2\tau\\
4\tau&0&6\tau&0&6&d_2\\
0&8\tau&0&16\tau&0&d_3-8\tau^2\\
16\tau^2&0&40\tau^2&0&40\tau&d_4\\
\lambda&0&\mu&0&\nu &d_{2s}\\
\end{array}\right]}.
$$

Noticing that $\nu\tau=\mu$, one has
$\det\overbar M=2^{14}\tau^3(\nu\tau-\mu)=0$. Hence,
since $x_0$ is also a solution to $\overbar Mx=\overbar P$, we must have
\begin{equation} \label{eq-n=3}
0=\det\overbar M_3=d_2(2-s)2^{2s+8}\tau^{s+2}+d_4(s-1)2^{2s+6}\tau^{s+1}-d_{2s}2^{10}\tau^3.
\end{equation}

Since $d_{2s}\ne 0$, the right-hand side of this equality is a polynomial, which forces $\tau$ to be constant.
This proves Proposition~\ref{prop-isoparametric-->constantangle}
in the case $n=3$.

Next, we treat the general case $n>3$. Our goal is to show that the augmented matrix
$[M\,|\,P]$ of size $(2n-1)\times 2n$ has the same key properties as in the cases
$n=2$ and $n=3$, which led to the polynomial identities \eqref{eq-n=2} (for $n=2$) and
\eqref{eq-detM3}–\eqref{eq-n=3} (for $n=3$). This is more involved; we establish the needed
facts in the appendix (Section~\ref{sec-appendix}) via a sequence of lemmas and propositions.
For the application here, we split the proof for $n>3$ into two parts.

\vt
\noindent
\textbf{Case 3:} $n>3$, $n$ \emph{even.}
\vt

By Proposition~\ref{prop-mainlinear}\,(i)(a), $\det M=0$. Since $x_0$ solves $Mx=P$, it follows that
$\det M_j=0$ for all $j\in\{1,\dots,2n-1\}$. Let $j=j_*$ be the index provided by
Proposition~\ref{prop-mainlinear}\,(i)(b). Then the equality $\det M_{j_*}=0$ is a nontrivial
algebraic equation in the variable $\tau$.

\vspace{.1cm}
\noindent
{\bf Case 4:} $n>3$, $n$ \emph{odd}.
\vt

By Lemma~\ref{lem:correction}, there exists $s_*\ge 2n$ such that $d_{s_*}\ne 0$.
Let $M(s_*)$ be the matrix obtained from $M$ by replacing its last row with $L_{s_*}$ (omitting
the last entry). By Proposition~\ref{prop-mainlinear}\,(ii)(a), we have $\det M(s_*)=0$. So, as before,
$$\det M_j(s_*)=0 \quad \forall j\in\{1,\dots,2n-1\}.$$

Finally, setting $j=n$, it follows from
Proposition~\ref{prop-mainlinear}\,(ii)(b) that the equality $$\det M_n(s_*)=0$$ yields a nontrivial
algebraic equation in $\tau$.

This concludes the proof of Proposition~\ref{prop-isoparametric-->constantangle}.
\end{proof}

\section{Proof of Theorem~\ref{thm-main}}

In this brief section we prove Theorem~\ref{thm-main},
which we restate here for the reader’s convenience:

\main*

\begin{proof}
(i) $\Rightarrow$ (ii). If $\Sigma$ is isoparametric, then by
Proposition~\ref{prop-isoparametric-->constantangle} it has constant angle;
hence, by the first part of Proposition~\ref{prop-constantangle}, it has constant principal curvatures.

\smallskip
\noindent
(ii) $\Rightarrow$ (iii). This follows directly from the second part of Proposition~\ref{prop-constantangle}.

\smallskip
\noindent
(iii) $\Rightarrow$ (iv) (for $\epsilon=-1$). In $\mathbb H^n$, a
complete hypersurface is isoparametric if and only if it is homogeneous.
Therefore any vertical cylinder in $\mathbb H^n\times\R$ over a complete
isoparametric hypersurface of $\mathbb H^n$ is homogeneous; slices are trivially
homogeneous; and parabolic bowls are homogeneous (as already verified).
Hence all hypersurfaces listed in (iii) are homogeneous in $\mathbb H^n\times\R$.

\smallskip
\noindent
(iv) $\Rightarrow$ (i). It is well known that this holds in general, i.e., any homogeneous
hypersurface of a Riemannian manifold is isoparametric.

\smallskip
\noindent
(iv) $\Leftrightarrow$ (v) (for $\epsilon=1$). The implication (v) $\Rightarrow$ (iv)
is immediate. Conversely, assume $\Sigma\subset \mathbb S^n\times\R$ is homogeneous.
From (iv) $\Rightarrow$ (i) $\Rightarrow$ (iii), $\Sigma$ must be one of the hypersurfaces in (a) or (b).
But a vertical cylinder  in $\s^n\times\R$
over a complete isoparametric hypersurface $\Sigma_0$ of $\s^n$ is homogeneous if and only
if $\Sigma_0$ is homogeneous. This proves the equivalence.
\end{proof}


\section{Appendix} \label{sec-appendix}

In this appendix, we prove some results which we have used in the proof of
Proposition~\ref{prop-isoparametric-->constantangle}. We keep the notation of
the preceding section.

To start with, define the row vector
\begin{equation}\label{LK}
\mathfrak L_{k-1}=(p_{k,0},\dots,p_{k,n-1},q_{k,0},\dots,q_{k,n-1}), \,\,\, k\ge 2,
\end{equation}
and denote by $Z=[p_{k,\ell},q_{k,\ell}]$ the $(2n-1)\times 2n$ matrix with
rows $\mathfrak L_{k-1}$, $2\le k\le2n$:
\begin{equation} \label{eq-Zmatrix}
Z=[p_{k,\ell},q_{k,\ell}]=
\begin{bmatrix}
\mathfrak L_1\\
\vdots\\
\mathfrak L_{2n-1}
\end{bmatrix}.
\end{equation}

Our interest in the matrix $Z$ relies on the equality
$Z=[-P_\tau,M],$ where
\begin{equation} \label{eq-ptau}
P_\tau:=P-P_d, \quad P_d=
\begin{bmatrix}
d_1\\
d_2\\
\vdots\\
d_{2n-1}
\end{bmatrix}.
\end{equation}

In what follows, we establish a series of lemmas that will lead to
a complete description of $Z$ (see Propositions~\ref{prop-coeff3} and~\ref{prop-p&q}).


\begin{lemma}\label{lem-reduction}
Let $\alpha_{\ell,k}$ and $\beta_{\ell,k}$ be the coefficient functions defined in~\eqref{eq-Dderivatives},
that is,
\begin{equation}\label{eq-Dderivatives02}
D^{(k)}(r)=\sum_{\ell=0}^{n-1}(\alpha_{\ell,k}+\beta_{\ell,k}r)s_\tau^\ell(r)c_\tau^{n-1-\ell}(r).
\end{equation}
Then, for any $n, k\in\N$, the following equalities hold:
$$
\left.
\begin{array}{lcll}
\alpha_{0,k+1}&=&\beta_{0,k}+\alpha_{1,k};\\[.7ex]
\alpha_{\ell,k+1}&=&\beta_{\ell,k}+(\ell+1)\alpha_{\ell+1,k}+\tau(n-\ell)\alpha_{\ell-1,k}&\text{if}\ \ell\in\{1,\dots,n-2\};\\[.7ex]
\alpha_{n-1,k+1}&=&\beta_{n-1,k}+\tau\alpha_{n-2,k};\\[.7ex]
\\
\beta_{0,k+1}&=&\beta_{1,k};\\[.7ex]
\beta_{\ell,k+1}&=&(\ell+1)\beta_{\ell+1,k}+\tau(n-\ell)\beta_{\ell-1,k}&\text{if}\ \ell\in\{1,\dots,n-2\};\\[.7ex]
\beta_{n-1,k+1}&=&\tau\beta_{n-2,k}.\\[.7ex]
\end{array}
\right.
$$

\end{lemma}
\begin{proof}
Considering~\eqref{eq-Dderivatives02}, we have that
\begin{eqnarray*}
D^{(k+1)}(r)&=& \sum_{\ell=0}^{n-1}\left[\beta_{\ell,k}s_\tau^\ell(r)c_\tau^{n-1-\ell}(r)+\ell(\alpha_{\ell,k}+\beta_{\ell,k}r)s_\tau^{\ell-1}(r)c_\tau^{n-\ell}(r)\right.\\
&&\qquad\left.+(n-1-\ell)(\alpha_{\ell,k}+\beta_{\ell,k}r)s_\tau^{\ell+1}(r)c_\tau^{n-2-\ell}(r)\right]\\
&=&\sum_{\ell=0}^{n-1}\beta_{\ell,k}s_\tau^\ell(r)c_\tau^{n-1-\ell}(r)+\sum_{p=0}^{n-1}p(\alpha_{p,k}+\beta_{p,k}r)s_\tau^{p-1}(r)c_\tau^{n-p}(r)\\
&&+\sum_{q=0}^{n-1}(n-1-q)(\alpha_{q,k}+\beta_{q,k}r)s_\tau^{q+1}(r)c_\tau^{n-2-q}(r)\\
&=&\sum_{\ell=0}^{n-1}\beta_{\ell,k}s_\tau^\ell(r)c_\tau^{n-1-\ell}(r)+\sum_{\ell=0}^{n-2}\ell(\alpha_{\ell+1,k}+\beta_{\ell+1,k}r)s_\tau^{\ell}(r)c_\tau^{n-1-\ell}(r)\\
&&+\sum_{\ell=1}^{n-1}(n-\ell)(\alpha_{\ell-1,k}+\beta_{\ell-1,k}r)s_\tau^{\ell}(r)c_\tau^{n-1-\ell}(r)\\
&=&(\beta_{0,k}+\alpha_{1,k}+\beta_{1,k}r)c_\tau^{n-1}(r)+(\beta_{n-1,k}+\tau\alpha_{n-2,k}+\tau\beta_{n-2,k}r)s_\tau^{n-1}(r)\\
&&+\sum_{\ell=1}^{n-2}[\beta_{\ell,k}+(\ell+1) (\alpha_{\ell+1,k}+\beta_{\ell+1,k}r)\\
&&\phantom{aaaaaaaaaaaaaaaaaaaa} +\tau(n-\ell)(\alpha_{\ell-1,k}+\beta_{\ell-1,k}r)]s_\tau^\ell(r)c_\tau^{n-1-\ell}(r).
\end{eqnarray*}

The result follows by comparison of the coefficients of the last equality with those of
$D^{(k+1)}$ as given in~\eqref{eq-Dderivatives02}.
\end{proof}


Now, with our purpose of describing the matrix $Z$,
we establish some fundamental properties of the coefficients $p_{k,\ell}$ and $q_{k,\ell}$ defined by
the equality
\begin{equation} \label{eq-linearcombination01}
\alpha_{0,k+1}=\sum_{\ell=0}^{n-1}(p_{k+1,\ell}\alpha_{\ell,0}+q_{k+1,\ell}\beta_{\ell,0}).
\end{equation}

We start with the case $n=3$, in which one has
\begin{equation}\label{eq-ic3}
p_{2,0}=2\tau, \,\, p_{2,1}=0, \,\, p_{2,2}=2, \,\, q_{2,0}=0, \,\, q_{2,1}=2, \,\,\text{and}\,\, q_{2,2}=0.
\end{equation}
Besides, for all $k\ge 2,$ we have from Lemma~\ref{lem-reduction} that
\begin{eqnarray*}
\alpha_{0,k+1}&=&\sum_{\ell=0}^2(p_{k,\ell}\alpha_{\ell,1}+q_{k,\ell}\beta_{\ell,1})\\[1ex]
&=&2\tau p_{k,1}\alpha_{0,0}+(p_{k,0}+\tau p_{k,2})\alpha_{1,0}+2p_{k,1}\alpha_{2,0}\\[1ex]
&&+(p_{k,0}+2\tau q_{k,1})\beta_{0,0}+(p_{k,1}+q_{k,0}+\tau q_{k,2})\beta_{1,0}+(p_{k,2}+2q_{k,1})\beta_{2,0}.
\end{eqnarray*}

Comparing this last equality with~\eqref{eq-linearcombination01} for
$n=3$, we conclude that,  for all $k\geq 2$, the following identities hold:
\begin{equation}\label{eq-rec3}
\left.\begin{array}{rcl}
p_{k+1,0}&=&2\tau p_{k,1}\\
p_{k+1,1}&=&p_{k,0}+\tau p_{k,2}\\
p_{k+1,2}&=&2p_{k,1}\\
\\
q_{k+1,0}&=&p_{k,0}+2\tau q_{k,1}\\
q_{k+1,1}&=&p_{k,1}+q_{k,0}+\tau q_{k,2}\\
q_{k+1,2}&=&p_{k,2}+2q_{k,1}.
\end{array}\right.
\end{equation}

By means of the  equalities in~\eqref{eq-rec3}, we can compute explicitly the coefficients
$p_{k,\ell}$ and $q_{k,\ell}$ as shown in the proposition below.

\begin{proposition}\label{prop-coeff3}
For any $k\geq 2$, $s\geq 1$ and $\ell\in\{0,1,2\}$, the following hold:
\begin{itemize}[parsep=1ex]
\item[\rm (i)] If $k+\ell$ is odd, then $p_{k,\ell}=0$, if $k+\ell$ is even, then $q_{k,\ell}=0$;
\item[\rm (ii)] $p_{2s,0}=\tau p_{2s,2}= 2^{2s-1}\tau^s$;
\item[\rm (iii)] $p_{2s+1,1}=2^{2s}\tau^s$;
\item[\rm (iv)] $q_{2s+1,0}=\tau q_{2s+1,2}= (2s+1)2^{2s-1}\tau^s$;
\item[\rm (v)] $q_{2s,1}=s2^{2s-1}\tau^{s-1}$.
\end{itemize}
\end{proposition}
\begin{proof}
(i) We proceed by induction on $k$. The result is true if $k=2$ by the initial conditions~\eqref{eq-ic3}.
Suppose now that for any $\ell$ such that $k+\ell$ is odd we have $p_{k,\ell}=0$ and
for any $\ell$ such that $k+\ell$ is even we have $q_{k,\ell}=0$. The first result follows by applying the inductive
hypothesis to the equalities for $p_{k+1,\ell}$ in~\eqref{eq-rec3}. Indeed, in these equalities, just the $p$-functions appear, and
the parity of the sum of the indices does not change. Similarly, in the equation for $q_{k+1,\ell}$,
the parity of the sum of the indices of the $q$-functions  doesn't change, whereas it changes for the $p$-functions.


\vspace{.1cm}
\noindent
(ii) We prove the two equalities separately.  By~\eqref{eq-rec3}, we have
$$
p_{2s+2,0}=2\tau p_{2s+1,1}=\tau p_{2s+2,2}.
$$
For the second equality, we proceed by induction on $s$. The case $s=1$
is true because of the initial conditions~\eqref{eq-ic3}. Now, let $s\geq 1$,
and suppose that $p_{2s,0}=2^{2s-1}\tau^s$. Applying~\eqref{eq-rec3} twice and considering the previous equality,
we have from  the inductive hypothesis that
\begin{eqnarray*}
\nonumber p_{2s+2,0}&=& 2\tau p_{2s+1,1}\\
&=&2\tau(p_{2s,0}+\tau p_{2s,2})\\
&=&2\tau (2^{2s-1}\tau^s+2^{2s-1}\tau^s)\ =\ 2^{2s-1}\tau^{s-1}.
\end{eqnarray*}
\vspace{.1cm}
\noindent
(iii) Fix $s\geq 1$. By~\eqref{eq-rec3} and the proved equality (ii),  we have
$$
p_{2s+1,1}=p_{2s,0}+\tau p_{2s,2}=2(2^{2s-1}\tau^s)=2^{2s}\tau^s.
$$
\vspace{.1cm}
\noindent
(iv) The first equality can be proved with the help of~\eqref{eq-rec3} and equality (ii):
$$
q_{2s+1,0}=p_{2s,0}+2\tau q_{2s,1}=\tau p_{2s,2}+\tau q_{2s,1}=\tau q_{2s+1,2}.
$$

For the second equality, we proceed by induction
on $s$. From~\eqref{eq-ic3} and~\eqref{eq-rec3}, we have:
$$
q_{3,0}=p_{2,0}+2\tau q_{2,1}=4\tau.
$$
Suppose the result is true for $s\geq 1$. By~\eqref{eq-rec3}, the previous equality,
and (ii), we have:
\begin{eqnarray*}
q_{2s+3,0} &=&p_{2s+2,0}+2\tau q_{2s+2,1}\\
&=&2^{2s+1}\tau^{s+1}+2\tau (p_{2s+1,1}+q_{2s+1,0}+\tau q_{2s+1,2})\\
&=&2^{2s+2}\tau^{s+1}+4\tau q_{2s+1,0}\\
&=&2^{2s+2}\tau^{s+1}+4\tau(2^{2s-1}\tau^{s-1}(2s+1))\\
&=&2^{2s+1}\tau^s(2s+3).
\end{eqnarray*}
\vspace{.1cm}
\noindent
(v) From~\eqref{eq-rec3}, (iii), and (iv),  we get
\begin{eqnarray*}
q_{2s,1}&=&p_{2s-1,1}+q_{2s-1,0}+\tau q_{2s-1,2}\\
&=& 2^{2s-2}\tau^{s-1}+2 q_{2s-1,0}\\
&=& 2^{2s-2}\tau^{s-1}+2((2s-1)2^{2s-3}\tau^{s-1})\\
&=&s2^{2s-1}\tau^{s-1},
\end{eqnarray*}
and this finishes the proof.
\end{proof}


Next, we apply Lemma~\ref{lem-reduction} to establish relations between the coefficient
functions of the expression of $\alpha_{0,k+1}$ as in~\eqref{eq-linearcombination01} for any $n\geq 4$.

\begin{lemma}\label{lem-coeff}
Let $p_{k+1,\ell}$ and $q_{k+1,\ell}$ be the coefficient functions defined in~\eqref{eq-linearcombination},
i.e.,
\begin{equation*}
\alpha_{0,k+1}=\sum_{\ell=0}^{n-1}(p_{k+1,\ell}\alpha_{\ell,0}+q_{k+1,\ell}\beta_{\ell,0}).
\end{equation*}
Then, for all  $n\geq 4$ and $k\geq 2$, the following equalities hold:
\begin{eqnarray*}
p_{2,0}&=&(n-1)\tau, \quad p_{2,2}\,=\,2\quad\text{\rm and}\quad p_{2,\ell}\ =\ 0 \ \text{\rm if}\  \ell\ne 0,2\\
p_{k+1,0}&=&(n-1)\tau p_{k,1}\\
p_{k+1,\ell}&=&\ell\cdot p_{k,\ell-1}+(n-1-\ell)\tau p_{k,\ell+1},\quad\text{\rm if}\ \ell\in\{1,\dots,n-2\}\\
p_{k+1,n-1}&=& (n-1)p_{k,n-2}\\
\\
q_{2,1}&=&2\quad\text{\rm and}\quad q_{2,\ell}\ =\ 0  \ \text{\rm if}\  \ell\ne 1\\
q_{k+1,0}&=&p_{k,0}+(n-1)\tau q_{k,1}\\
q_{k+1,\ell}&=&p_{k,\ell}+\ell\cdot q_{k,\ell-1}+(n-1-\ell)\tau q_{k,\ell+1},\quad\text{\rm if}\ \ell\in\{1,\dots,n-2\}\\
q_{k+1,n-1}&=& {p_{k,n-1}+(n-1)q_{k,n-2}}
\end{eqnarray*}
\end{lemma}
\begin{proof}
We have that
$$
\alpha_{0,k+1}=\sum_{\ell=0}^{n-1}(p_{k+1,\ell}\alpha_{\ell,0}+q_{k+1,\ell}\beta_{\ell,0}),
$$
but, by Lemma~\ref{lem-reduction} we can also prove that
\begin{eqnarray*}
\alpha_{0,k+1}&=&\sum_{\ell=0}^{n-1}(p_{k,\ell}\alpha_{\ell,1}+q_{k,\ell}\beta_{\ell,1})\\
&=&p_{k,0}(\beta_{0,0}+\alpha_{1,0})+p_{k,n-1}(\beta_{n-1,0}+\tau\alpha_{n-2,0})\\
&&+\sum_{\ell=1}^{n-2}p_{k,\ell}(\beta_{\ell,0}+(\ell+1)\alpha_{\ell+1,0}+(n-\ell)\tau\alpha_{\ell-1,0})\\
&&+q_{k,0}\beta_{1,0}+q_{k,n-1}\tau\beta_{n-2,0}\\
&&+\sum_{\ell=1}^{n-2}q_{k,\ell}((\ell+1)\beta_{\ell+1,0}+(n-\ell)\tau\beta_{\ell-1,0})\\
&=&(n-1)\tau p_{k,1}\alpha_{0,0}+(p_{k,{0}}+(n-2)\tau p_{k,2})\alpha_{1,0}\\
&&+\sum_{\ell=2}^{n-3}(\ell\cdot p_{k,\ell-1}+(n-1-\ell)\tau p_{k,\ell+1})\alpha_{\ell,0}\\[.5ex]
&&+((n-2)p_{k,n-3}+\tau p_{k,n-1})\alpha_{n-2,0}+({(n-1) p_{k,n-2}})\alpha_{n-1,0}\\[.5ex]
&&+(p_{k,0}+(n-1)\tau q_{k,1})\beta_{0,0}+(p_{k,1}+q_{k,0}+(n-2)\tau q_{k,2})\beta_{1,0}\\
&&+\sum_{\ell=2}^{n-3}(p_{k,\ell}+\ell\cdot q_{k,\ell-1}+(n-1-\ell)\tau q_{k,\ell+1})\beta_{\ell,0}\\
&&+(p_{k,n-2}+(n-2)q_{k,n-3}+\tau q_{k,n-1})\beta_{n-2,0}\\
&&+({p_{k,n-1}+(n-1)q_{k,n-2}})\beta_{n-1,0}
\end{eqnarray*}
In the last equality, if $n=4$, then the summand from $\ell=2$ to $\ell=n-3$ should be ignored.
The result follows by comparison of the corresponding coefficients.
\end{proof}
\begin{example} \label{examp-Zmatrices}
By means of the relations established in Lemma~\ref{lem-coeff}, we can obtain the matrix
$Z=[p_{k,\ell},q_{k,\ell}]$ for any value of $n\ge 2$. For $n$ from $2$ to $5$,  for instance,
$Z$ is given by the matrices:

\vspace{.3cm}
\noindent
$\bullet \,\,
{\small
\begin{bmatrix}
\tau & 0 & 0 & 2 \\[.5ex]
0 & \tau & 3\tau & 0 \\[.5ex]
\tau^2 &0 & 0 & 4\tau
\end{bmatrix}};$

\vspace{.3cm}
\noindent
$\bullet \,\,
\left[
\begin{smallmatrix}
2\tau &0&2&0&2&0\\[.5ex]
0&4\tau&0&6\tau&0&6\\[.5ex]
8\tau^2&0&8\tau&0&16\tau&0\\[.5ex]
0&16\tau^2&0&40\tau^2&0&40\tau\\[.5ex]
32\tau^3&0&32\tau^2&0&96\tau^2&0
\end{smallmatrix}\right];$

\vspace{.3cm}

\noindent
$\bullet \,\,
\left[
\begin{smallmatrix}
3\tau&0&2&0&0&2&0&0\\[.5ex]
0&7\tau&0&6&9\tau&0&6&0\\[.5ex]
21\tau^2&0&20\tau&0&0&28\tau&0&24\\[.5ex]
0&61\tau^2&0&60\tau&105\tau^2&0&100\tau&0\\[.5ex]
183\tau^3&0&182\tau^2&0&0&366\tau^2&0&360\tau\\[.5ex]
0&547\tau^3&0&546\tau^2&1281\tau^3&0&1274\tau^2&0\\[.5ex]
1641\tau^4&0&1640\tau^3&0&0&4376\tau^3&0&4368\tau^2
\end{smallmatrix}\right];$

\vspace{.3cm}

\noindent
$\bullet \,\, \left[
\begin{smallmatrix}
4\tau & 0          & 2            & 0          & 0           & 0           & 2            & 0           & 0           &   0         \\[.5ex]
0 & 10\tau     & 0            & 6          & 0           & 12\tau      & 0            & 6           & 0           &   0         \\[.5ex]
40\tau^2 &  0          & 32\tau       & 0          & 24          & 0           & 40\tau       & 0           & 24          &   0         \\[.5ex]
0 & 136\tau^2  & 0            & 120\tau    & 0           & 200\tau^2   & 0            & 160\tau     & 0           & 120         \\[.5ex]
544\tau^3& 0          & 512\tau^2    & 0          & 480\tau     & 0           & 816\tau^2    & 0           & 720\tau     &   0         \\[.5ex]
0&2080\tau^3 & 0            & 2016\tau^2 & 0           & 3808\tau^3  & 0            & 3584\tau^2  & 0           & 3360\tau    \\[.5ex]
8320\tau^4&0          & 8192\tau^3   & 0          & 8064\tau^2  & 0           & 16640\tau^3  & 0           & 16128\tau^2 & 0           \\[.5ex]
0&32896\tau^4& 0            & 32640\tau^3& 0           & 74880\tau^2 & 0            & 73728\tau^3 & 0           & 72576\tau^2 \\[.5ex]
131584\tau^5&0          & 131072\tau^4 & 0          & 130560\tau^3& 0           & 328960\tau^2 & 0           & 326400\tau^3& 0
\end{smallmatrix}\right].
$
\end{example}


In the next proposition, we establish some  fundamental properties of
the coefficients $p_{k,\ell}$, $q_{k,\ell}$ that will lead to a general description
of the matrix $Z$. These properties can be checked in the matrices of the above example.

\begin{proposition}\label{prop-p&q}
Given $n\geq 4$ and $k\geq 2$, the following assertions hold:
\begin{itemize}[parsep=1ex]
\item[\rm (i)]  $p_{k,\ell}=0$ if $k+\ell$ is odd, and $q_{k,\ell}=0$ if $k+\ell$ is even;
\item[\rm (ii)] $p_{k,\ell}=q_{k,\ell}=0$ for all $\ell>k$;
\item[\rm (iii)] $p_{k,k}=k!$ and {$q_{k+1,k}=(k+1)!$} for all $k\ge 2$;
\item[\rm (iv)] if $k-\ell=2s$, then $p_{k,\ell}=\sigma_{k,\ell}^m(n)\tau^s$,
where $\sigma_{k,\ell}^m$ is a  polynomial of degree $m\ge s$ with positive leading coefficient;
\item[\rm (v)] if $k-\ell=2s+1$, then $q_{k,\ell}=\omega_{k,\ell}^m(n)\tau^s$, where
$\omega_{k,\ell}^m$ is a polynomial of degree $m\ge s$ with positive leading coefficient.
\end{itemize}
\end{proposition}
\begin{proof}
(i) Analogous to the one given for Proposition~\ref{prop-coeff3}-(i).

\vspace{.1cm}
\noindent
(ii) We proceed by induction on $k$. The result is true if $k=2$, by the initial conditions in Lemma~\ref{lem-coeff}.
Suppose now that, for any $\ell>k$, we have $p_{k,\ell}=0$. In this setting, for
any $\ell>k+1$, we have from Lemma~\ref{lem-coeff} that $p_{k+1,\ell}=a p_{k,\ell-1}+b \tau p_{k,\ell+1}$,
where $a$ and $b$ are positive integers. Since $\ell+1>\ell-1>k$, it follows from  the inductive hypothesis
that $p_{k+1,\ell}=0$. The assertion on $q_{k+1,\ell}$ is proved analogously by using induction
and the result just proved for $p_{k+1,\ell}$.

\vspace{.1cm}
\noindent
(iii) We proceed by induction on $k$. For $k=2$ the result is immediate.
Suppose that it is  true for $k\ge 2$. Under these conditions,
one concludes that the result is true for $k+1$ by applying the recursive formulas
of Lemma~\ref{lem-coeff} and the proved item~(ii).

\vspace{.1cm}
\noindent
(iv) We proceed by induction on $k$: suppose that for any $\ell$ such that $k-\ell=2s$ the result is true. Then,
\begin{eqnarray*}
p_{k+1,\ell+1} &=& (\ell+1)p_{k,\ell}+(n-2-\ell)\tau p_{k,\ell+2}\\
            &=& (\ell+1)\sigma_{k,\ell}^{m_1}(n)\tau^s+(n-2-\ell)\tau\sigma_{k,\ell+2}^{m_2}(n)\tau^{s-1}\\
            &=& \sigma_{k+1,\ell+1}^m(n)\tau^s,
\end{eqnarray*}
where $\sigma_{k+1,\ell+1}^m$ is the polynomial of degree $m=\max\{m_1,m_2\}>s$ defined by
$$\sigma_{k+1,\ell+1}^m=(\ell+1)\sigma_{k,\ell}^{m_1}+(n-2-\ell)\sigma_{k,\ell+2}^{m_2}.$$
Clearly, the leading coefficient of $\sigma_{k+1,\ell+1}^m$ is positive. Therefore, the result
is true for $k+1.$

\vspace{.1cm}
\noindent
(v) Analogous to (iv).
\end{proof}


Now, we proceed to determine the rank of the  matrix
$Z=[-P_\tau|M]$ defined in~\eqref{eq-Zmatrix}.
It will be convenient to reinterpret the recursive formulas 
of Lemma \ref{lem-coeff} in vectorial form. With that in mind, 
we  consider the decomposition $\R^{2n}=\R^n\times\R^n$ and  
set $e_1=(1,\dots, 0)\in\R^n$. From the definition \eqref{LK} of the row vectors  
$\mathfrak L_k$, after a straightforward computation we have that the 
row vectors $\mathfrak L_k$ relate as
$$\mathfrak L_{k}=\mathfrak L_{k-1} Q=(e_1,0)Q^{k+1}, \,\,\, k\ge 2,$$

\noindent where  $Q$ is the $2n\times2n$ matrix defined by
\begin{equation}\label{kac+}
Q=\left[\begin{array}{c|c}
\mathcal K&I_n\\[1ex]
\hline
\\[-1.5ex]
O_n& \mathcal K
\end{array}\right],
\end{equation}
being the blocks $0_n$ and $I_n$  the null and identity $n\times n$ matrices, respectively,
and $\mathcal K=(k_{ij})$  the $n\times n$ matrix defined by the equalities
$$
k_{ij}=\left\{\begin{array}{cl}
(n-j)\tau&\text{if}\ j=i-1,\\[1ex]
i&\text{if}\ j=i+1,\\[1ex]
0&\text{otherwise},
\end{array}\right.
$$
that is,
\begin{equation*}
\mathcal K=\left[\begin{array}{ccccccc}
0&1&0&&0&0&0\\
(n-1)\tau&0&2&\cdots&0&0&0\\
0&(n-2)\tau&0&&0&0&0\\
&\vdots&&\ddots&&\vdots\\
0&0&0&&0&n-2&0\\
0&0&0&\cdots&2\tau&0&n-1\\
0&0&0&&0&\tau&0
\end{array}\right].
\end{equation*}

\begin{remark}
In the case $\tau=1$, the transpose of $\mathcal K$ is known as the \emph{Kac} or
\emph{Sylvester-Kac matrix}.
Due to that nomenclature, for $n\ge 2$ and $\tau\ne 0$, we shall call $\mathcal K$ the $\tau$-\emph{Kac matrix}
of order $n$. We add that the Kac matrix appears in many different contexts as, for example,
in the description of  random walks on a hypercube (see~\cite{edelman-kostlan} and the references therein).
\end{remark}

\begin{lemma}\label{kac-lemma}
The $\tau$-Kac matrix  $\mathcal K$ of order $n$ has the following properties:
\begin{itemize}[parsep=1ex]
\item[\rm (i)] It has $n$ simple eigenvalues $\lambda_0,\dots ,\lambda_{n-1}$, which are
$$\lambda_{\ell}=(n-1-2\ell)\sqrt\tau, \quad \ell\in\{0,1,\dots,n-1\}. $$
In particular $\lambda_{\ell}$ is real if $\tau>0$, and purely imaginary if $\tau<0$.
\item[\rm (ii)] Its rank is $n$, if  $n$ is even, and $n-1$ if $n$ is odd. In particular,
$\mathcal K$ is nonsingular if and only if $n$ is even.
\item[\rm (iii)] The coordinates of $e_1=(1,0,\dots,0)\in\R^n$ with respect to
the basis of its eigenvectors  are all different from zero.
\end{itemize}
\end{lemma}
\begin{proof}
We shall show (i) through  a straight adaptation of the beautiful proof  given for~\cite[Theorem 2.1]{edelman-kostlan}.
To that end, we first define the  functions (seen as vectors)
\[
\omega_\ell(x)=s_\tau^\ell(x)c_\tau^{n-1-\ell}(x), \quad \ell\in\{0,\dots,n-1\},
\]
where $s_{\tau}$ and $c_{\tau}$ are the functions defined in \eqref{def-sc}.
Clearly, the set $\mathcal B=\{\omega_0,\dots,\omega_{n-1}\}$ is linearly independent, and so
it generates a vector space $V$ of dimension $n$. In addition, we have that
$$\frac{d}{d x}\omega_\ell(x)=\ell\omega_{\ell-1}(x)+(n-1-\ell)\tau\omega_{\ell+1},$$
from which we conclude that $d/dx$ is an operator on $V$ whose matrix with respect to the basis
$\mathcal B$ is the  $\tau$-Kac matrix $\mathcal K$.

Now, considering the equality
\[
e^{(n-1-2\ell)\sqrt\tau x}=(c_\tau(x)+\sqrt\tau s_\tau(x))^{n-1-\ell}(c_\tau(x)-\sqrt\tau s_\tau(x))^\ell,
\]
we have that, for each $\ell\in\{0,\dots,n-1\}$,
the function $x\mapsto e^{(n-1-2\ell)\sqrt\tau x}$
belongs to $V$, and it is an eigenvector of ${d}/{dx}$ with eigenvalue $(n-1-2\ell)\sqrt\tau$.
This proves (i).

The statement (ii) follows directly from (i).

Finally, to prove (iii), we identify $\R^n$ with $V,$ so that $e_1$
becomes the first vector  $\omega_0(x)=c_\tau^{n-1}(x)$ of $\mathcal B$.
Since $c_\tau(x)=\frac 12\left(e^{\sqrt\tau x}+e^{-\sqrt\tau x}\right)$,
we have from the binomial formula that
$$c_\tau^{n-1}(x)=\frac1{2^{n-1}}\sum_{\ell=0}^{n-1} \binom{n-1}{\ell}e^{(n-1-2\ell)\sqrt\tau x},$$
which clearly proves~(iii).
\end{proof}

Let $\{v_0,\dots, v_{n-1}\}\subset\R^n$ be the basis of eigenvectors of
the $\tau$-Kac matrix $\mathcal K$. Consider the decomposition $\R^{2n}=\R^n\times\R^n$ and
define the following vectors:
\begin{equation} \label{eq-xl&yl}
x_\ell=(v_\ell,0) \quad\text{and}\quad y_\ell=(0,v_\ell), \,\,\, \ell\in\{0,\dots, n-1\}.
\end{equation}

As a direct consequence of
Lemma~\ref{kac-lemma} and the block structure of the $2n\times 2n$ matrix $Q$ defined in~\eqref{kac+}, we have:

\begin{corollary} \label{kac-cor}
The following assertions hold true.
\begin{itemize}[parsep=1ex]
\item[\rm (i)] The matrix $Q$ is nonsingular if and only if $n$ is even.
\item[\rm (ii)] For any $\ell\in\{0,\dots,n-1\}$, the vectors $x_\ell$ and $y_\ell$ defined in~\eqref{eq-xl&yl}
satisfy:
$$
x_\ell Q=\lambda_\ell x_\ell+y_\ell \quad\text{and}\quad y_\ell Q=\lambda_\ell y_\ell,
$$
i.e. $x_\ell$ is a generalized eigenvector of $Q$, whereas $y_\ell$ is an eigenvector of $Q$.
\item[\rm (iii)] Regarding the coordinates of $e_1=(e_1,0)\in\R^{2n}$ with respect to the basis
$$\mathcal B=\{x_0,\dots, x_{n-1}, y_0,\dots, y_{n-1}\},$$
those with respect to the generalized eigenvectors
$x_\ell$ never vanish, whereas the ones with respect to the eigenvectors
$y_\ell$ are all zero.
\end{itemize}
\end{corollary}

Next, we establish a result that
plays a crucial role in the proof of Proposition~\ref{prop-mainlinear}.
In its proof, we shall consider a certain type of generalized Vandermonde matrix as defined
in~\cite{kalman}. First, let us recall that,
given $n$ pairwise distinct numbers
$x_0,\dots ,x_{n-1}$, the $n\times n$ Vandermonde matrix
$V(x_0,\dots,x_{n-1})$ is defined by
\[
V(x_0,\dots,x_{n-1})=
\begin{bmatrix}
1 & x_0 & x_0^2 & x_0^3 & \dots & x_0^{n-1}\\[1ex]
1 & x_1 & x_1^2 & x_1^3 & \dots & x_1^{n-1}\\[1ex]
\vdots & \vdots & \vdots & \vdots & \vdots & \vdots\\[1ex]
1 & x_{n-1} & x_{n-1}^2 & x_{n-1}^3 & \dots & x_{n-1}^{n-1}
\end{bmatrix},
\]
and its determinant is given by
\[
\det V(x_0,\dots, x_{n-1})=\prod_{i<j}(x_j-x_i),
\]
which implies that $V(x_0,\dots,x_{n-1})$ is nonsingular.

The generalized Vandermonde matrix
$V_2(x_0,\dots, x_{n-1})$ of type $2$ is the following  $2n\times 2n$ matrix:
\[
V_2(x_0,\dots,x_{n-1})=
\begin{bmatrix}
1 & x_0 & x_0^2 & x_0^3 & \dots & x_0^{2n-1}\\[1ex]
0 & 1 & 2x_0 & 3x_0^2 & \dots & (2n-1)x_0^{2n-2}\\[1ex]
1 & x_1 & x_1^2 & x_1^3 & \dots & x_1^{2n-1}\\[1ex]
0 & 1 & 2x_1 & 3x_1^2 & \dots & (2n-1)x_1^{2n-2}\\[1ex]
\vdots & \vdots & \vdots & \vdots & \vdots & \vdots\\[1ex]
1 & x_{n-1} & x_{n-1}^2 & x_{n-1}^3 & \dots & x_{n-1}^{2n-1}\\[1ex]
0 & 1 & 2x_{n-1} & 3x_{n-1}^2 & \dots & (2n-1)x_{n-1}^{2n-2}
\end{bmatrix}.
\]

Notice that each $x_i$ defines a pair of row vectors
$L_i$, $L_{i+1}$ which satisfy
\[
L_{i+1}=\frac{\partial L_{i}}{\partial x_i}\cdot
\]

It can be proved that (see~\cite{kalman})
\[
\det V_2(x_0,\dots, x_{n-1})=\prod_{i<j}(x_j-x_i)^4>0,
\]
which implies that $V_2(x_0,\dots, x_{n-1})$ is nonsingular as well.

\begin{proposition}\label{prop-crucialrole}
Setting $e_1=(e_1,0)\in\R^{2n}$,
the following assertions hold.
\begin{itemize}[parsep=1ex]
\item[\rm (i)] If \,$n$ is even, for any positive integer $s$, the set
$$\{e_1Q^{s}, e_1Q^{s+1}, \dots ,e_1Q^{s+2n-1}\}$$
is linearly independent.
\item[\rm (ii)] If \,$n$ is odd, let $s\geq 2n$ and define
$$\Lambda=\{e_1Q^{2}, e_1Q^{3},\dots,e_1Q^{2n-1}\}, \quad \Lambda_s=\Lambda\cup\{e_1Q^s\}.$$
Then, denoting by
$C_1, \dots ,C_{2n}$ the column vectors of the matrix $Z(s)$ whose rows are
the vectors of $\Lambda_s$, the  following hold:
\begin{itemize}[parsep=1ex]
\item[\rm(a)] $\Lambda$ is linearly independent, whereas $\Lambda_s$ is linearly dependent;
\item[\rm(b)] $C_1$ is in the span of the odd columns $C_3, C_5, \dots, C_n$;
\item[\rm(c)] $C_{n+1}$ is in the span of the even columns $C_{n+3}, C_{n+5},\dots,C_{2n}$.
\end{itemize}
\end{itemize}
\end{proposition}
\begin{proof}
(i) When $n$ is even, we know from Corollary~\ref{kac-cor}-(i) that $Q$ is invertible. So, it
suffices to prove~(i) for $s=0$. Notice that, for any $\ell\in\{0,\dots, n-1\},$
one has $x_\ell Q^2=\lambda_\ell^2x_\ell+2\lambda_\ell y_\ell$.
Therefore, by induction, one has that the equality
\begin{equation} \label{eq-Qpowerk}
x_\ell Q^k=\lambda_\ell^kx_\ell+k\lambda_\ell^{k-1}y_\ell
\end{equation}
holds  for any positive integer $k$.

Now, consider the following vector equation of variables $c_0,\dots, c_{2n-1}$:
\begin{equation} \label{eq-vectorequation}
\sum_{k=0}^{2n-1}c_ke_1Q^k=0.
\end{equation}
We have from Corollary~\ref{kac-cor}-(iii) that
\begin{equation} \label{eq-e1}
e_1=\sum_{\ell=0}^{n-1}a_\ell x_\ell,
\end{equation}
with $a_\ell\ne 0$ for any $\ell\in\{0,\dots, n-1\}$.

Setting $\overbar x_\ell=a_\ell x_\ell$ and $\overbar y_\ell=a_\ell y_\ell$, we get from
equalities~\eqref{eq-Qpowerk}--\eqref{eq-e1} that
\[
\sum_{\ell=0}^{n-1}\sum_{k=0}^{2n-1}\left(\lambda_\ell^kc_k\overbar x_\ell+k\lambda_\ell^{k-1}c_k\overbar y_\ell\right)=0,
\]
which implies that~\eqref{eq-vectorequation} is equivalent to the homogeneous linear system
of equations:
\begin{equation}  \label{eq-system}
\sum_{k=0}^{2n-1}\lambda_\ell^kc_k=0,\quad\sum_{k=0}^{2n-1}k\lambda_\ell^{k-1}c_k=0,\quad \ell\in\{0,\dots,n-1\}.
\end{equation}

The matrix of coefficients of the system~\eqref{eq-system} is
the generalized Van\-der\-mon\-de matrix $V_2(\lambda_0,\dots,\lambda_{n-1})$,
and so it is invertible, since the eigenvalues
$\lambda_0, \dots, \lambda_{n-1}$ are pairwise distinct. Thus, $c_k=0$ for all $k\in\{0,\dots, 2n-1\}$, which proves~(i).

\vspace{.1cm}
\noindent
(ii)-(a) Assume now that $n$ is odd and consider the following
vector equation of $2n-1$ variables $c_2,\dots, c_{2n-1}, c_s$:
\begin{equation} \label{eq-vectorequation02}
\sum_{k=2}^{2n-1}c_ke_1Q^k+c_se_1Q^s=0.
\end{equation}

Proceeding as in the case $n$ even, we conclude that the equation~\eqref{eq-vectorequation02}
is equivalent to the linear system:
\begin{equation}  \label{eq-system02}
\sum_{k=2}^{2n-1}\lambda_\ell^kc_k+\lambda_\ell ^sc_s=0,\quad\sum_{k=2}^{2n-1}k\lambda_\ell^{k-1}c_k+s\lambda_\ell^{s-1}c_s=0, \quad \ell\in\{0,\dots,n-1\}.
\end{equation}

Since $n$ is odd, we have from Corollary~\ref{kac-cor} that  $\lambda_{(n-1)/{2}}=0$,
so that~\eqref{eq-system02} is a homogeneous linear system of $2n-2$ equations with $2n-1$ unknowns.
The matrix $R$ of coefficients of the system~\eqref{eq-system02} can be decomposed into $n-1$ blocks as
\[
R=\left[
\begin{array}{c}
B_0\\[1ex]
B_1\\
\vdots\\
B_\ell\\
\vdots\\
B_{n-1}
\end{array}\right], \quad \ell\neq\frac{n-1}{2},
\]
where the generic block
$B_\ell$ is the $2\times(2n-1)$ matrix given by
\[
B_\ell=
\begin{bmatrix}
\lambda_\ell^2&\lambda_\ell^3&\cdots&\lambda_\ell^{2n-1}&\lambda_\ell^s\\[1ex]
2\lambda_\ell&3\lambda_\ell^2&\cdots&(2n-1)\lambda_\ell^{2n-2}&s\lambda_\ell^{s-1}
\end{bmatrix}.
\]

It is immediate that ${\rm rank}\,R\le 2n-2$,
proving that $\Lambda_s$ is linearly dependent. To prove that $\Lambda$ is linearly independent,
consider the matrix $R_{2n-1}$ obtained from $R$ by
removing the last of its $2n-1$ columns.
Each block of $R_{2n-1}$ can be modified
through elementary operations on its rows, so that
the row-equivalent resulting matrix $\overbar R_{2n-1}$ is composed
by $n-1$ blocks of $2\times(2n-2)$ matrices of the type
\[
\begin{bmatrix}
1&\lambda_\ell&\lambda_\ell^2&\cdots&\lambda_\ell^{2n-3}\\[1ex]
0&1&2\lambda_\ell&\cdots&(2n-3)\lambda_\ell^{2n-4}
\end{bmatrix},
\]
from which we conclude that $\overbar R_{2n-1}$ is the  generalized Vandermonde
matrix
$$V_2(\lambda_0,\dots,\lambda_{(n-3)/2},\lambda_{(n+1)/2},\dots, \lambda_{n-1}).$$

Therefore, $\overbar R_{2n-1}$ is nonsingular, which implies that
${\rm rank}\,R_{2n-1}=2n-2,$ as we wished to prove.

\vspace{.1cm}
\noindent
(ii)-(b) Set $n=2m+1$.
The matrix composed by the rows $e_1Q^2,\dots,e_1Q^{2n-1}, e_1Q^s$ has $2n$ columns. Since we
are only interested in the first $n$, it will be convenient to work
in $\mathbb R^n$ by identifying $x_\ell$ with $v_\ell$.  We also point out that  $C_1,\ \dots\ ,C_n$
are the columns of the matrix whose rows are $e_1\mathcal K^2,\dots,e_1\mathcal K^{2n-1}, e_1\mathcal K^{s}$,
$s\ge 2n$. As in the above argument, the last row will be immaterial. So, without loss of generality,
we can assume $s=2n$, in which case $C_1,\dots, C_n$ are nothing but the first $n$ columns of the matrix
$Z$ defined in~\eqref{eq-Zmatrix}.

We claim that the span of the set $\{C_{2i+1}\}_{i=0,\dots,m}$ has dimension $m$.
From the considerations of  the preceding paragraph,
it suffices to show that the span of the rows $\{e_1\mathcal K^{2i}\}_{i\in\{1,\dots m+1\}}$
has dimension $m$. Indeed, observing the zero entries of these row vectors as
determined in Proposition~\ref{prop-p&q}, we have to consider just the odd rows,
since we are only interested in the odd columns (cf.~the matrices $Z$ in
Example~\ref{examp-Zmatrices} for the cases $n=3,5$).

As before, we consider the vector equation of
$m+1$ variables $c_1,\dots, c_{m+1}$:
\begin{equation*}
\sum_{j=1}^{m+1}c_je_1\mathcal K^{2j}=0,
\end{equation*}
which is equivalent to the linear system of $2m$ equations with $m+1$ unknowns:
\begin{equation}  \label{eq-system03}
\sum_{j=1}^{m+1}\lambda_\ell^{2j}c_j=0,\quad \ell\in\{0,\dots, n\}.
\end{equation}

The $\ell$-th row of the matrix $R$ of coefficients of the  system~\eqref{eq-system03} is
$$
\left[\begin{array}{cccc}
\lambda_\ell^2&
\lambda_\ell^{4}&
\cdots&
\lambda_\ell^{2m+2}
\end{array}\right].
$$
In addition,  by Lemma~\ref{kac-lemma}, $\lambda_m=0$ and, for any eigenvalue $\lambda\ne 0,$ $-\lambda$
is an eigenvalue as well. Therefore, each one of these rows appears twice, showing that the rank of $R$
is at most $m$.

Now, consider  the $m$ rows which are distinct from each other. Extract the factor $\lambda_\ell^2$
from each one of them and ignore the last column. The resultant matrix is easily seen to
be an $m\times m$ standard Vandermonde matrix, and so it is nonsingular.
Therefore, the rank of $R$ is $m$ and the claim is proved.

To conclude the proof of (ii)-(b), we have just to observe that,
by Proposition~\ref{prop-p&q},
all entries of $Z=[p_{k,\ell},q_{k,\ell}]$ in the
$(2n-1)\times n$ block of the coefficients
$p_{k,\ell}$ which lie above the ``factorial diagonal'' (i.e., the one of entries
$2!, 3!, \dots, (n-1)!$) are zero. Indeed,  this property of
$Z$ clearly implies that the set $\{C_{2i+1}\}_{i=1,\dots,m}$ is linearly independent.
Consequently, $C_1$ is in the span of $\{C_{2i+1}\}_{i=1,\dots,m}$, since
the span of $\{C_{2i+1}\}_{i=0,\dots,m}$ has dimension $m$.

\vspace{.1cm}
\noindent

(ii)-(c) Similarly, we claim that the span of the set $\{C_{n+2i+1}\}_{i=0,\dots,m}$ has dimension $m$.
We start the proof of this claim by noting that, for any $j$, we have
\begin{equation*}
Q^j=\left[\begin{array}{c|c}
\mathcal K^j&j\mathcal K^{j-1}\\[1ex]
\hline
\\[-1.5ex]
0_n& \mathcal K^j
\end{array}\right].
\end{equation*}

As before, we can work in $\mathbb R^n$ and assume $s=2n$. In this setting,
$C_{n+1},\ \dots\ ,C_{2n}$ are the columns of the matrix whose rows are $2e_1\mathcal K,3e_1\mathcal K^2,\dots,2ne_1\mathcal K^{2n-1}$. Since now
we are interested just in even columns,
it is enough to consider just even rows. Therefore, this time the starting vectorial equation is
$$
\sum_{j=1}^{m+1}(2j+1)e_1\mathcal K^{2j}=0,
$$
which leads to the homogeneous linear system whose matrix of coefficients is composed by rows of the type
$$
\left[\begin{array}{cccc}
3\lambda_\ell^2&
5\lambda_\ell^{4}&
\cdots&
(2m+3)\lambda_\ell^{2m+2}
\end{array}\right].
$$

Once again, since just even powers of $\lambda_\ell$ appear, by Lemma~\ref{kac-lemma},
we have exactly one zero row
(that one for $\ell=m$) and that each non-zero row appears twice. Considering then the $m$ rows which are
distinct from each other, we get the  matrix
$$
\left[\begin{array}{cccc}
3\lambda_0^2&
5\lambda_0^{4}&
\cdots&
(2m+3)\lambda_0^{2m+2}\\
\vdots&\vdots&\ddots&\vdots\\
3\lambda_{m-1}^2&
5\lambda_{m-1}^{4}&
\cdots&
(2m+3)\lambda_{m-1}^{2m+2}
\end{array}\right].
$$
For any $\ell$ and $j,$ extract the factors $\lambda^2_\ell$
from the $\ell$-th row and $2j+1$ from the $j$-th column.
In this way, we reduce this matrix to the standard Vandermonde matrix,
showing that it has maximal rank $m$.
\end{proof}

\begin{lemma}\label{lem-powers}
Given $m\in\mathbb N$, $i, j\in\{1,\dots,m\},$ and $h_{ij},b_i,c_j\in\mathbb R$,
consider  the $m\times m$ matrices
$X=(h_{ij}\tau^{b_i+c_j})_{i,j}$ and $Y=(h_{ij})_{i,j}$. Then,
$$\det X=\tau^{(\sum_{i=1}^nb_i+\sum_{j=1}^nc_j)}\det Y.$$
\end{lemma}
\begin{proof}
For any $i, j\in\{1,\dots,m\}$, extract $\tau^{b_i}$ from the $i$-th row and $\tau^{c_j}$ from the $j$-th column,
then apply the standard properties of the determinant.
\end{proof}

\begin{remark} \label{remk-lemma}
Any square submatrix of $Z$ can be written as the matrix
$X$ in the statement of Lemma~\ref{lem-powers}. Moreover, for all indices $i$ and $j$,
both $b_i$ and $c_j$ can be written as $\gamma/2$, where $\gamma$ is a nonnegative
integer. In particular, any minor of $Z$ is either zero or a monomial in $\tau$ of the form
$\mu\tau^{\gamma/2}$,  $\mu\ne 0$.
An easy way to see that is by writing the zeros of $Z$ as a product $0\tau^{p/q}$, where $p$ and $q$ are suitable integers.
Then, we have that
\[
b_i=\frac{i-1}{2}, \quad c_j=\gamma_1(j)/2,
\]
where $\gamma_1(j)/2$ is the power of $\tau$
in the first entry of the $j$-th column.
For instance, when $n=4$,  we can verify this property by rewriting the matrix $Z$ as
(see Example~\ref{examp-Zmatrices}):
$$
\begin{bmatrix}
3\tau&0\tau^{\frac 12}&2&0\tau^{-\frac 12}&0\tau^{\frac 12}&2&0\tau^{-\frac 12}&0\tau^{-1}\\[.5ex]
0\tau^{\frac 32}&7\tau&0\tau^{\frac 12}&6&9\tau&0\tau^{\frac 12}&6&0\tau^{-\frac 12}\\[.5ex]
21\tau^2&0\tau^{\frac 32}&20\tau&0\tau^{\frac 12}&0\tau^{\frac 32}&28\tau&0\tau^{\frac 12}&24\\[.5ex]
0\tau^{\frac 52}&61\tau^2&0\tau^{\frac 32}&60\tau&105\tau^2&0\tau^{\frac 32}&100\tau&0\tau^{\frac 12}\\[.5ex]
183\tau^3&0\tau^{\frac 52}&182\tau^2&0\tau^{\frac 32}&0\tau^{\frac 52}&366\tau^2&0\tau^{\frac 32}&360\tau\\[.5ex]
0\tau^{\frac 72}&547\tau^3&0\tau^{\frac 52}&546\tau^2&1281\tau^3&0\tau^{\frac 32}&1274\tau^2&0\tau^{\frac 32}\\[.5ex]
1641\tau^4&0\tau^{\frac 72}&1640\tau^3&0\tau^{\frac 52}&0\tau^{\frac 72}&4376\tau^3&0\tau^{\frac 52}&4368\tau^2
\end{bmatrix}.
$$
\end{remark}

Recall that $M(s)$ denotes the $(2n-1)\times(2n-1)$ matrix obtained from the
matrix $Z(s)$ defined in the statement of
Proposition~\ref{prop-crucialrole}-(ii) by exclusion of
its first column.

\begin{proposition} \label{prop-mainlinear}
The matrices $M$ and $M(s)$ have the following properties:
\begin{itemize}[parsep=1ex]
\item[\rm (i)] If $n$ is even, one has:
\begin{itemize}[parsep=1ex]
\item[\rm (a)] $M$ has rank $2n-2$;
\item[\rm (b)] there exists  $j_*\in\{1,\dots,2n-1\}$ such that
$$
\det M_{j_*}=\mu_0\tau^{{\gamma_0}}+\sum_{i=1}^{2n-1}\mu_id_i\tau^{{\gamma_i}},
$$
where $\mu_0\neq0,\mu_1,\dots,\mu_{2n-1}$ and $\gamma_0>\gamma_1>\dots>\gamma_{2n-1}>0$
are all integers.
\end{itemize}
\item[\rm (ii)] If $n$ is odd, for any $s\geq 2n$, one has:
\begin{itemize}[parsep=1ex]
\item[\rm (a)] $M(s)$ has rank $2n-2$;
\item[\rm (b)] the determinant of $M_n(s)$ is given by
$$
\det M_n(s)=\mu_sd_s\tau^{{\gamma_s}}+\sum_{i=1}^{2n-2}\mu_id_i\tau^{{\gamma_i}},
$$
where $\mu_1,\dots,\mu_{2n-2},\mu_s\neq 0$ and $\gamma_1>\dots>\gamma_{2n-2}>\gamma_s>0$ are all integers.
\end{itemize}
\end{itemize}
\end{proposition}

\begin{proof}
(i)-(a) Looking at the position of the zero entries of $M$ as given in Proposition~\ref{prop-p&q}-(i),
one easily sees that its odd rows define a set of
$n$ vectors spanning an $(n-1)$-dimensional vector space. Hence, they are linearly dependent, that is,
the rank of $M$ is at most $2n-2$. However, Proposition~\ref{prop-crucialrole}-(i) for $s=2$
gives that $Z$ has rank $2n-1$, so that the rank of $M$ is exactly $2n-2$.

\vspace{.1cm}
\noindent
(i)-(b) Denote by $M_j^\tau$ (resp. $M_j^d$) the matrix obtained from $M$ by replacing its
$j$-th column with the column matrix $P_\tau$ (resp. $P_d$) defined in~\eqref{eq-ptau}.
Since, as seen above, $Z=[-P_\tau, M]$ has maximal rank $2n-1$,
there exists  $j_*\in\{1,\dots,2n-1\}$ such that $\det M_{j_*}^{\tau}\neq 0$.
By applying Lemma~\ref{lem-powers} to the matrix $M_{j_*}^{\tau}$ (see also Remark~\ref{remk-lemma}), we conclude
that there exist integers  $\mu_0\neq 0$ and  $\gamma_0>0$
such that $\det M_{j_*}^{\tau}=\mu_0\tau^{\frac{\gamma_0}{2}}$.

With the notation of Lemma~\ref{lem-powers}, we have for the matrix $X=M$ that
\begin{equation}\label{eq-bc}
2b_i=i-1,\quad 2c_j=\left\{\begin{array}{ll}
2-j&\text{if}\ j\in\{1,\dots,n-1\};\\[1ex]
n+1-j&\text{if}\ j\in\{n,\dots,2n-1\}.
\end{array}\right.
\end{equation}

To obtain the constants $\mu_1,\dots,\mu_{2n-1}$ and $\gamma_1,\dots,\gamma_{2n-1}$ as in the
statement, it suffices to
expand $\det M_{j_*}^d$ with respect to its $j_*$-th column, and then apply Lemma~\ref{lem-powers}
to any minor occurring in such Laplace expansion. Due to the relations~\eqref{eq-bc}, it is clear that the
exponents $\gamma_i=\gamma_i(j_*)$ are all integers. The same is true for the coefficients $\mu_i$,
since any entry of the matrix $M$ is either an integer or an integer
multiple of some power of $\tau$.

It follows from the first equality in~\eqref{eq-bc} that
the function $i\mapsto b_i$ is strictly increasing,
so that  $i\mapsto\gamma_i$
is strictly decreasing (notice that the power $\gamma_i/2$ comes from the minor
that suppresses the $i$-th row of $M$). It is also clear that $\gamma_0=\gamma_1+1$.
Finally, considering all values of $j$ and writing $\gamma_i=\gamma_i(j)$,
we have from the second equality in~\eqref{eq-bc}
that the smallest value of $\gamma_{2n-1}(j)$ is achieved  for $j=1$ or $j=n$.

Taking all that into account, we have
\begin{eqnarray*}
{\gamma_{2n-1}}&=&{\gamma_{2n-1}(j_*)} \,\geq\, 2\sum_{i=1}^{2n-2}b_i+2\sum_{j=2}^{2n-1}c_j\\
&=&\frac12\left[(2n-2)(2n-3)-(n-2)(n-3)-n(n-3)\right]\\
&=&{n(n-1)}>0.
\end{eqnarray*}

\vspace{.1cm}
\noindent
(ii)-(a) It follows directly from items (a) and (b) of Proposition~\ref{prop-crucialrole}-(ii).

\vspace{.1cm}
\noindent
(ii)-(b) Defining $M_j(s)$, $M_j^\tau(s)$ and $M_j^d(s)$ analogously to
$M_j$, $M_j^\tau$ and $M_j^d$,  we have from
Proposition~\ref{prop-crucialrole}-(ii) that
$\det M_j^\tau(s)=0$ for all $j,$ in particular, for $j=n$. So,
$$\det M_n(s)=\det M_n^d(s).$$

Henceforth, proceeding just as in the proof of
(i)-(a), we get the equality for $\det M_n$ as in the statement.
It only remains to prove that $\mu_s\neq 0$. However,  $\mu_s$
is the determinant of the submatrix of $M(s)$ obtained
by removing its last row and its $n$-th column (be aware of the ordering of the columns:
the $n$-th column of $M(s)$ is the $(n+1)$-th column of $Z(s)$) and,
by Proposition~\ref{prop-crucialrole}-(ii), this submatrix is non-singular.
\end{proof}

\end{document}